\documentclass{amsart}
\usepackage{mathrsfs}
\usepackage[all]{xy}

\usepackage{mathtools}
\usepackage{mathbbol}
\usepackage{wasysym}
\usepackage{amssymb}
\usepackage{MnSymbol}
\usepackage{comment}
\usepackage{enumerate}

\usepackage{ushort}

\usepackage{changepage}
\usepackage{algorithm}
\usepackage[noend]{algpseudocode}
\makeatletter
\def\BState{\State\hskip-\ALG@thistlm}
\makeatother

\usepackage{ifpdf}
\ifpdf 
  \usepackage[pdftex]{graphicx}
  \DeclareGraphicsExtensions{.pdf,.png,.jpg,.jpeg,.mps}
  \usepackage{pgf}
\else 
  \usepackage{graphicx}
  \DeclareGraphicsExtensions{.eps,.bmp}
  \usepackage{pgf}
  \usepackage{pstricks}
\fi
\usepackage{epic,bez123}
\usepackage{wrapfig}

\newtheorem{thm}{Theorem}[section]
\newtheorem{prop}[thm]{Proposition}
\newtheorem{cor}[thm]{Corollary}
\newtheorem{lem}[thm]{Lemma}
\newtheorem{conj}[thm]{Conjecture}
\newtheorem*{claim}{Claim}

\newtheorem{mainthm}{}

\newtheorem*{mainthmB}{Theorem B}

\newtheorem*{conv}{Convention}

\theoremstyle{definition}
\newtheorem{defn}[thm]{Definition}

\theoremstyle{remark}
\newtheorem*{quest}{Question}
\newtheorem*{rem}{Remark}
\newtheorem{example}[thm]{Example}
\newtheorem*{examples}{Examples}

\newtheorem*{ack}{Acknowledgment}

\newcommand{\isom}{\mbox{Isom}} 
 
\newcommand{\Gx }{\mathscr{G} (G, S)}

\newcommand{\pGf}{\partial_\lambda{G}}

\newcommand{\pX}{\partial{\mathrm X}}
\newcommand{\pXf}{\partial_\lambda{X}}

\newcommand{\PS}[1]{\mathcal P_{#1}(s)}

\newcommand{\e}[1]{\omega(#1)}

\newcommand{\ax}{\mathrm{Ax}}

\newcommand{\path}[1]{\mathbf{\ushort{#1}}}

\newcommand{\diam }[1]{{\textbf{diam}\big(#1\big)}}

\newcommand{\proj}{\textbf{d}^\pi}

\newcommand{\len }{\ell}

\DeclarePairedDelimiter\ceil{\lceil}{\rceil}

\usepackage[bookmarks=true, pdfauthor={YANG Wenyuan}]{hyperref}

\newtheoremstyle{query}%
{}{}
{\color{red}}
{}
{\sffamily\bfseries}{:}{12pt}
{}
\theoremstyle{query}

\calclayout
\begin{document}

\title[Statistically convex-cocompact actions]{Statistically convex-cocompact actions of groups with contracting elements}

\author{Wen-yuan Yang}

\address{Beijing International Center for Mathematical Research (BICMR), Beijing University, No. 5 Yiheyuan Road, Haidian District, Beijing, China}
 
\email{yabziz@gmail.com}
\thanks{}



\date{July 2, 2017}

\dedicatory{}


\begin{abstract}
This paper presents a study of the asymptotic geometry of groups with contracting elements, with   emphasis on a subclass of \textit{statistically convex-cocompact} (SCC) actions. The class of SCC actions includes relatively hyperbolic groups,   CAT(0) groups with rank-1 elements  and mapping class groups, among others. We exploit an extension lemma to prove that a group with SCC actions contains  large free sub-semigroups,  has purely exponential growth and contains a class of \textit{barrier-free} sets with a growth-tight property.  Our study produces new results and recovers existing ones  for many interesting groups through a unified and  elementary approach. 
\end{abstract}

\maketitle

\setcounter{tocdepth}{1} \tableofcontents

\section{Introduction}
\subsection{Background}
Suppose that a group $G$ admits a proper and  isometric action on a proper 
geodesic metric space $(\mathrm Y, d)$. The group $G$ is assumed to be \textit{non-elementary}: there is no finite-index subgroup isomorphic to the integers   $\mathbb Z$ or to the trivial group.  The goal of this paper is to study   the asymptotic geometry of the group action  in the presence of a contracting element.   Through a  unified approach, we shall present a series of applications  to the following   classes of groups  with contracting elements:
  
\begin{itemize}
\item 
${\mathbb{Hyp}}=\big\{$ a group $G$ acts properly and cocompactly on  a  $\delta$-hyperbolic   space $(\mathrm Y, d)$ $\big\}$. See \cite{Gro}, \cite{GH} for general references.

\item 
$\mathbb{{RelHyp}}=\big\{$ a relatively hyperbolic group $G$ acts on  a Cayley graph $(\mathrm Y, d)$   with respect to a generating set $S$  $\big\}$. See \cite{Bow1}, \cite{Osin}, and \cite{DruSapir}. 
 
\item 
$\mathbb{{Floyd}}=\big\{$ a group $G$ with non-trivial Floyd boundary acts on  a Cayley graph $(\mathrm Y, d)$   with respect to a generating set $S$ $\big\}$. See \cite{Floyd}, \cite{Ka}, \cite{Ge1}, \cite{GePo2}. 

\item 
$\mathbb{CAT_0^1}=\big\{$ a group $G$ acts properly and cocompactly on a CAT(0) space $(\mathrm Y, d)$ with rank-1 elements   $\big\}$. See \cite{BF2}， \cite{Ballmann}, \cite{BriHae}. 

\item 
$\mathbb{{GSC}}=\big\{$ a $\mathrm{Gr}'(1/6)$-labeled graphical small cancellation group $G$  with finite components  labeled by a finite set $S$ acts on the   Cayley graph $(\mathrm Y, d)$  with respect to the generating set $S$  $\big\}$. See  \cite{ACGH},  \cite{Gro4}.

\item
$\mathbb{{Mod}}=\big\{$ the mapping class group $G$ of a closed orientable surface with genus greater than two acts on   Teichm\"{u}ller space $(\mathrm Y, d)$ equipped with the Teichm\"{u}ller metric  $\big\}$. See \cite{FLP},    \cite{FMbook}.

\end{itemize}


Let us give a definition of a contracting element (cf. Definition \ref{ContrDefn}). First, a subset $X$ is called \textit{contracting} if any metric ball disjoint with $X$ has a uniformly bounded projection to $X$ (cf. \cite{Minsky}, \cite{BF2}).  An element $g \in G$ of infinite order is \textit{contracting} if for some basepoint $o\in \mathrm Y$, an orbit $\{g^n\cdot o: n\in \mathbb Z \}$ is contracting, and the map $n\in \mathbb Z \to h^no\in \mathrm Y$ is a quasi-isometric embedding. The prototype of a contracting element is  a hyperbolic isometry on hyperbolic spaces, but more interesting examples are furnished by the following:
\begin{itemize}
\item
hyperbolic elements in $\mathbb{{RelHyp}}$ and  $\mathbb{{Floyd}}$, cf. \cite{GePo4}, \cite{GePo2};
\item
rank-1 elements in  $\mathbb{CAT_0^1}$, cf. \cite{Ballmann}, \cite{BF2};
\item
certain infinite order elements in  $\mathbb{GSC}$, cf. \cite{ACGH};
\item
pseudo-Anosov elements in  $\mathbb{{Mod}}$, cf. \cite{Minsky}.
\end{itemize}
 
We shall demonstrate in  \textsection \ref{SSProper} that  a proper action with a contracting element already  admits interesting consequences.  
With appropriate hypothesis, more information can be obtained on a subclass of \textit{statistically convex-cocompact actions} (SCC) to be discussed in \textsection \ref{SSSCC}. This is a central concept of the paper since    it allows to generalize   dynamical aspects of the theory of convex-cocompact Kleinian groups to the above list ($\mathbb{Hyp}-\mathbb{Mod}$) of groups.

Following the groundbreaking work of Masur and Minsky \cite{MinM} \cite{MinM2}, the study of mapping class groups from the point of view of geometric group theory has received much attention. Indeed, one of the motivations behind the present study is the application of the approach presented here to  $\mathbb {Mod}$. Most   mapping class groups are not relatively hyperbolic, cf. \cite{BDM}. Nevertheless, their action on Teichm\"{u}ller spaces exhibits very interesting negative behavior, thanks to  the following two facts from Teichm\"{u}ller geometry:
\begin{enumerate}
\item
Minsky's result that a pseudo-Anosov element is contracting \cite{Minsky};
\item
the fact that the group action of mapping class groups on Teichm\"{u}ller space is SCC, which follows from a deep theorem of Eskin, Mirzakhani and Rafi \cite[Theorem 1.7]{EMR}, as observed in \cite[Section 10]{ACTao}.
\end{enumerate} 
This study therefore considers some applications in $\mathbb {Mod}$. We emphasize that most of arguments, once SCC actions are provided,  are completely general without appeal to specific theory of groups   ($\mathbb{Hyp}-\mathbb{Mod}$) under consideration.

We are now formulating the setup of main questions to be addressed in next subsections. 
Fixing a basepoint $o \in \mathrm Y$, one expects to read off information from the growth of orbits in the ball of radius $n$:
$$N(o, n):=\{g\in G: d(o, go)\le n\}.$$ For instance, a celebrated theorem of Gromov in
\cite{Gro2} says that  the class of  virtually nilpotent groups is characterized by the polynomial growth of $N(o, n)$ in Cayley graphs. Most groups that one encounters  in fact admit  exponential growth: for instance, thanks to the existence of non-abelian free subgroups. This is indeed the case in our study, where a well-known ping-pong game played by contracting elements gives rise to many free subgroups. In this regard, an  asymptotic quantity called the \textit{critical exponent} (also called the \textit{growth rate}) is associated with the growth function. In practice, it is useful to consider the critical exponent $\e \Gamma$ for a \textit{subset} $\Gamma \subset G$:
\begin{equation}\label{criticalexpo}
\e \Gamma = \limsup\limits_{n \to \infty} \frac{\log \sharp(N(o, n)\cap \Gamma)}{n},
\end{equation}
which is independent of the choice of $o \in \mathrm Y$. The following alternative definition of $\e \Gamma$ is very useful in counting problems:
$$\limsup_{n\to \infty}\frac{\log \sharp \left(A(o, n, \Delta)\cap \Gamma\right)}{n} = \e \Gamma,$$
where we consider the annulus $$A(o, n, \Delta)=\{g\in G: |d(o, go)-n|\le\Delta\}$$
for $\Delta>0$.

The remainder of this introductory section is based  on consideration of the following questions:
\begin{quest}
\begin{enumerate}
\item
When is the critical exponent $\e \Gamma$ (\ref{criticalexpo}) a true limit? Can the value $\e \Gamma$ be realized by some geometric subgroup?
\item
Under what conditions does a group action have so-called \textit{purely exponential growth}, i.e.,  there exists a constant $\Delta>0$ such that 
$$
\sharp A(o, n, \Delta) \asymp \exp( \e G n) \quad ?
$$
\item
Which subsets $X$ in $G$ are \textit{growth-tight}: $\e X<\e G$? This shall admit several applications to genericity problems studied in a subsequent paper \cite{YANG11}.   
\end{enumerate}
\end{quest}
 

\subsection{Large free semigroups} \label{SSProper} At  first glance, it seems somewhat of a leap of faith to anticipate that a general and non-trivial result can be obtained for   a proper group action   with a contracting element. Therefore, our first objective is to convince the reader that it is indeed  fruitful to work with this aim in mind. We are  going to introduce two general and interesting results that recover  many existing results in a unified manner. The first group of results concerns the existence of large free semi-subgroups in various interesting classes of groups.

To be concrete, let us motivate our discussion by considering the class of Schottky groups among   discrete groups in $\isom(\mathbb H^n)$, called \textit{Kleinian groups} in the literature. By definition,  a Schottky group is a free, convex-cocompact, Kleinian group in $\isom(\mathbb H^n)$.

A seminal work of Phillips and Sarnak \cite{PhiSarnak} showed that the critical exponent  of any \textit{classical} Schottky group in $\isom(\mathbb H^n)$ is uniformly bounded away from $n-1$ for $n> 3$, with the case $n=3$  being established later by Doyle \cite{Doyle}.  In   $n$-dimensional quaternionic hyperbolic spaces, a well-known result of Corlette \cite{Corlette}  implies that  the critical exponent of any subgroup in lattices is at most $4n$, uniformly different from the value $4n+2$   for lattices.  Recently, Bowen \cite{Bowen2} proved an extension of Phillips and Sarnak's work in higher even dimensions, showing that the critical exponent  of any free discrete group in $\isom(\mathbb H^{2n})$ for $n\ge 2$ is uniformly bounded away from $2n-1$. However,   he  also proved in \cite{Bowen1} that the $\pi_1$ of a closed hyperbolic 3-manifold contains a free subgroup with critical exponent arbitrarily close to $2$. Let us close this   discussion by mentioning a very recent result of Dahmani, Futer and Wise \cite{FutWise}  that for free groups $\mathbb F_n$ ($n\ge 2$) there exists a sequence of finitely generated subgroups with critical exponents tending to $\log(2n-1)$. Hence, an intriguing question arises concerning the conditions under which there exists a  gap of  critical exponents for free subgroups of   ambient groups. Although this question remains unanswered for  free groups, if we consider   the class of free semigroups, then we are indeed able to obtain a general result.

\begin{mainthm}[large free semigroups]\label{LargeFreeThm}
Let $G$ admit a proper action on a geodesic   space $(\mathrm Y, d)$ with a contracting element.   Fix a basepoint $o\in \mathrm Y$.  Then there exists a sequence of    free semigroups $\Gamma_n \subset G$ such that 
\begin{enumerate}
\item
$
\e {\Gamma_n} < \e G \text{ but } \e {\Gamma_n} \to \e G
$
as $n\to \infty$. 
\item
The standard Cayley graph of $\Gamma_n$ is sent by a natural orbital map to a quasi-geodesic tree rooted at $o$ such that each branch is a contracting quasi-geodesic.
\end{enumerate}
\end{mainthm}
\noindent The property in Assertion (1) will be referred to as the \textit{large free semigroup} property.  

The notion of quasi-convexity have been studied in the literature, some of which we shall consider in the sequel. Lets first introduce the strong one. A  subset $X$ is called \textit{$\sigma$-quasi-convex} for a function $\sigma: \mathbb R_+ \to \mathbb R_+$ if, given $c \ge
1$,  any $c$-quasi-geodesic with endpoints in $X$ lies in the neighborhood $N_{\sigma(c)}(X)$. It is clear that this notion of quasi-convexity is preserved up to a finite Hausdorff distance. A subset $\Gamma$ of $G$ is called \textit{$\sigma$-quasi-convex} if  the set $\Gamma\cdot o$ is  $\sigma$-quasi-convex   for some (or any) point $o \in \mathrm Y$.

We usually speak of a \textit{purely contracting} sub-semigroup $\Gamma$   if every non-trivial element is  contracting. Thus, the theorem can be rephrased as follows:
\begin{cor}
Any proper action on a geodesic   space with contracting element has  purely contracting, quasi-convex,  large free semigroups. 
\end{cor}

Construction of large quasi-geodesic trees can be traced back at least to the work of Bishop and Jones \cite{BishopJ} on the Hausdorff dimension of limit sets of Kleinian groups. Generalizing earlier results of  Patterson \cite{Patt} and Sullivan \cite{Sul}, they constructed such trees to give a lower bound on the Hausdorff dimension   in a very general setting. Later, their construction was   implemented for discrete groups acting on $\delta$-hyperbolic spaces by Paulin \cite{Paulin}.

Except the growth-tightness property, Mercat \cite{mercat} has proved the other properties of $\Gamma_n$ in \ref{LargeFreeThm} for (semi)groups acting properly on a hyperbolic space.  It should be noted here that all of these constructions  make   essential use of Gromov's hyperbolicity and its consequences, whence a direct generalization \`a la Bishop and Jones fails in non-hyperbolic spaces, such as  Cayley graphs of a relatively hyperbolic group. In this relative case, the present author \cite{YANG7} has introduced the notion of a ``partial'' cone, which allows one to construct  the desired trees by iterating partial cones.  This approach has been  applied to large quotients and Hausdorff dimensions  in \cite{YANG7}  and \cite{PYANG}.  Nevertheless, the current construction given in \textsection \ref{Section3} follows more closely the construction \`a la Bishop and  Jones.

Our method here relies only on the existence of a contracting element, which is usually thought of as a hyperbolic direction, a very partial negative curvature in a metric space. This sole assumption allows   our result to be applied to many more interesting class of groups: see the list at the beginning of this paper.   In the interests of clarity and brevity, we   mention  below some corollaries that we believe to be of particular interest.

\begin{thm}[$\mathbb {RelHyp}$]
A relatively hyperbolic group $ G\in \mathbb{RelHyp}$ has quasi-convex, large free  sub-semigroups. 
\end{thm}
This result does not follow from Mercat's theorem, since the Cayley graphs of a relatively hyperbolic group are not  $\delta$-hyperbolic. Further applications of this corollary to Hausdorff dimension will be given elsewhere; cf. \cite{PYANG}.
\\ 

\paragraph{\textbf{The class of irreducible subgroups}} Given a proper action of $G$ on a geodesic metric space $(\mathrm Y, d)$, it is natural to look at an  \textit{irreducible}  subgroup $\Gamma$  which, by definition, contains at least two independent contracting elements (cf. \textsection \ref{SScontracting}). Equivalently, it is the same as a non-elementary subgroup  with at least one contracting element.  

In the context of mapping class groups, a sufficiently large subgroup was  studied first by McCarthy and Papadopoulos \cite{McPapa}, as an analog of non-elementary subgroups of Kleinian groups. By definition, a \textit{sufficiently large} subgroup is  one with at least two independent pseudo-Anosov elements, so  it coincides with the notion of an irreducible subgroup by  Lemma \ref{pAContr}. We refer the reader to \cite{McPapa} for a detailed discussion. The interesting examples include convex-cocompact  subgroups  \cite{FarbMosher}, the handlebody group, among many others. Hence, we deduce the following from \ref{LargeFreeThm}:
 
\begin{thm}[$\mathbb{Mod}$]
Consider a sufficiently large subgroup $\Gamma$ of $ G\in \mathbb{Mod}$. Then there exists a sequence of purely pseudo-Anosov, free semi-subgroups $\Gamma_n \subset \Gamma$ such that 
$$
\e {\Gamma_n} < \e \Gamma \text{ but } \e {\Gamma_n} \to \e \Gamma
$$
as $n\to \infty$. 
\end{thm}


\subsection{Statistically convex-cocompact actions}\label{SSSCC}
In this subsection, we focus on a subclass of proper actions akin to a cocompact action in a statistical sense.     By abuse of language, a geodesic between two sets $A$ and $B$ is a geodesic $[a, b]$ between $a\in A$ and $b\in B$.

A significant part of this paper is concerned with  studying  a \textbf{statistical} version of convex-cocompact actions. Intuitively, the reason that an action fails to be convex-cocompact is the existence of \textit{a concave region} formulated as follows. Given constants $0\le M_1\le M_2$, let $\mathcal O_{M_1, M_2}$ be the set  of elements $g\in G$ such that there exists some geodesic $\gamma$ between $B(o, M_2)$ and $B(go, M_2)$ with the property that the interior of $\gamma$ lies outside $N_{M_1}(Go)$; see Figure \ref{figure1}. 

\begin{figure}[htb] 
\centering \scalebox{0.6}{
\includegraphics{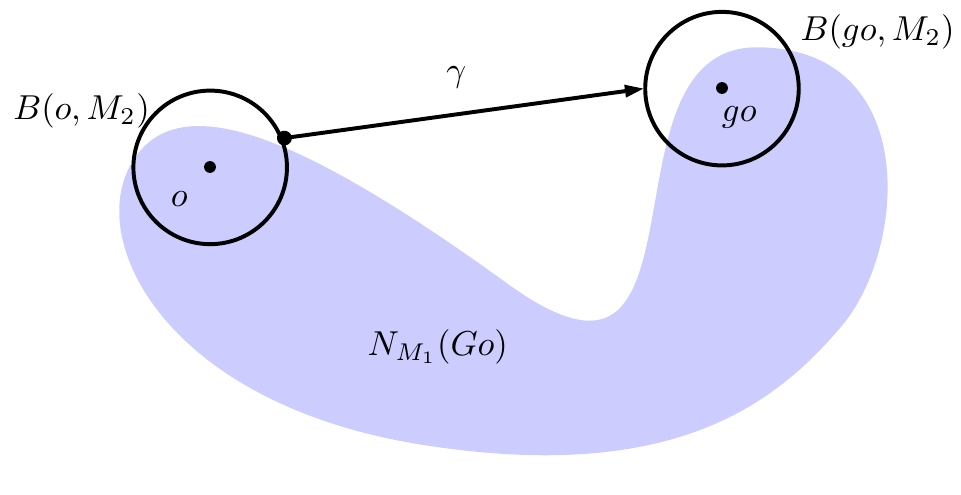} 
} \caption{Definition of $\mathcal O_{M_1, M_2}$} \label{figure1}
\end{figure}

\begin{defn}[statistically convex-cocompact action]\label{StatConvex}
If there exist two positive constants  $M_1, M_2>0$ such that $\e {\mathcal O_{M_1, M_2}} < \e G$, then the action of $G$ on $\mathrm Y$ is called \textit{statistically convex-cocompact (SCC)}.  

\end{defn}

\begin{rem}
\begin{enumerate}
\item
This underlying concept   was introduced by Arzhantseva  et al. in \cite{ACTao} as a generalization of a parabolic gap condition due to Dal'bo  et al. \cite{DOP}. We propose the particular terminology used here since  we shall prove (in a forthcoming work) that a SCC action with contracting elements is statistically hyperbolic, a  notion introduced earlier by Duchin  et al. \cite{DLM}.

\item (about parameters)
We have chosen two parameters $M_1, M_2$ (differing from \cite{ACTao})  so that the definition is flexiable and easy to verify.  In practice, however, it is  enough to take $M_1=M_2$ without losing anything, since $\mathcal O_{M_2, M_2} \subset \mathcal O_{M_1, M_2}$. Henceforth, we set $\mathcal O_M:=\mathcal O_{M, M}$ for ease of notation.  
\item
Moreover, when the SCC action contains a contracting element,  the definition will be proven independent of the basepoint in Lemma \ref{SCCBasepoint}.
\end{enumerate}
\end{rem}

The list of groups in ($\mathbb {Hyp}- \mathbb {Mod}$) all admit SCC actions on the interesting spaces. Here we emphasize some prototype examples motivating the notion of SCC actions. 
\begin{examples}
\begin{enumerate}
\item
Any proper and cocompact action on a geodesic metric space.  In this case, $\mathcal O_M$ is empty.
 
\item
The class of relatively hyperbolic groups with parabolic gap property (cf. \cite{DOP}); in this case, the set $\mathcal O_M$ is the union of finitely many parabolic groups, up to a finite Hausdorff distance. 
\item
The action of mapping class groups on Teichm\"{u}ller spaces is SCC (cf. \cite{ACTao}).  
\end{enumerate}
\end{examples}

Analogous to irreducible subgroups in a proper action,  it appears to be interesting to study the class of \textit{statistically convex-cocompact subgroups} (SCC subgroups) in a given SCC action. This should be regarded as a generalization of convex-cocompact subgroups   studied in some groups.

The notion of convex-cocompact subgroups in $\mathbb {Mod}$ was introduced by Farb and Mosher \cite{FarbMosher}.  This   is conceived  as an analog of the well-studied class of convex-cocompact Kleinian groups, with applications to surface group extensions.  We believe that SCC subgroups are a useful generalization of their notion from a dynamical point of view. Moreover,  the notion of SCC subgroups is strictly bigger than that of convex-cocompact subgroups in $\mathbb {Mod}$:
 
\begin{prop}[cf. \ref{ModSubSCC}]
In mapping class groups, there exist, free and non-free, non-convex-cocompact subgroups which admit    SCC actions on Teichm\"{u}ller spaces. 
\end{prop}

It is an interesting question to determine  to which extent a SCC subgroup generalizes the notion of a geometrically finite subgroup in the following classes of groups:
\begin{enumerate}
\item
In hyperbolic groups, does there exist a SCC subgroup which is not quasiconvex?

\item
In relatively hyperbolic groups, does there exist a SCC subgroup which is not relative quasiconvex?  The point is that whether the concave region $\mathcal O_M$ is always coming from the union of finitely many parabolic subgroups.

\end{enumerate}

The second main result of this paper concerns the quantitative  behavior of the orbital growth function. A group action has so-called \textit{purely exponential growth} if  there exists $\Delta>0$ such that 
$$
\sharp A(o, n, \Delta) \asymp \exp( \e G n)
$$
for  any $n\ge 1$. This property admits several interesting applications, for instance, to  statistical hyperbolicity \cite{DLM} \cite{OYANG}, to counting conjugacy classes and automatic structures \cite{AntCio}, and  to  the finiteness of Bowen--Margulis--Sullivan measure \cite{Roblin}\cite{YANG8}.

\begin{mainthm}[exponential growth] \label{StatisticThm}
Let $G$ admit a proper action on a geodesic   space $(\mathrm Y, d)$ with a contracting element. Then the following holds:
\begin{enumerate}
\item
The critical exponent is a true limit:
$$
\e \Gamma = \lim\limits_{n \to \infty} \frac{\log \sharp(N(o, n)\cap \Gamma)}{n}.
$$
\item
For some $\Delta>0$, we have
$$
\sharp A(o, n, \Delta) \prec \exp( \e G n)
$$ 
for any $n\ge 1$.
\item
If the action is SCC, then $G$ has  purely exponential growth.
 
\end{enumerate}
\end{mainthm}
\begin{rem}
Assertion (1) was established by Roblin \cite{Roblin2} for CAT($-1$) spaces, whose method used conformal density in a crucial way. For Kleinian groups, the proof of Assertion (2)  via the use of  Patterson--Sullivan measures is well known,  with the most general form being due to Coornaert \cite{Coor} for a discrete group acting on $\delta$-hyperbolic spaces.

In \cite{Roblin}, Roblin showed the equivalence of purely exponential growth and finiteness of Bowen--Margulis measure for discrete groups on CAT($-1$) spaces. In a coarse setting, we gave in \cite{YANG8} a characterization of purely exponential growth  in the setting of cusp-uniform actions on $\delta$-hyperbolic spaces. 

We emphasize that our elementary proof  does not use the machinery  of Patterson--Sullivan theory and, more importantly,  it is valid in a very general setting. 

\end{rem}

First of all, we give some  corollaries to the class of an irreducible subgroup of the first two Assertions in the theorem.  

The following application appears to be new even in $\mathbb {Hyp}$,   which  was proved recently for free groups by Olshanskii \cite{Ol2}. Indeed, an infinite subgroup of a hyperbolic group always contains a hyperbolic element that is contracting.
\begin{cor} 
Assume that $G$ acts properly on a geodesic   space $(\mathrm Y, d)$. Then the limit  
$$\lim\limits_{n \to \infty} \frac{\log \sharp(N(o, n)\cap \Gamma)}{n}$$
exists for any irreducible  subgroup $\Gamma$.
\end{cor}

Moreover, any irreducible subgroup $\Gamma$ has an upper exponential growth function as well:
$$
\sharp (A(o, n, \Delta) \cap \Gamma) \prec \exp( \e \Gamma n)
$$ 
for some $\Delta>0$. Specializing to the class of mapping class groups, the notion of an irreducible subgroup coincides with a sufficiently large subgroup.    Therefore, we obtain the following result which appears to be new:
\begin{cor}[$\mathbb {Mod}$]
For any sufficiently large subgroup $\Gamma$ of $G\in \mathbb {Mod}$, we have  
$$
\sharp (A(o, n, \Delta) \cap \Gamma) \prec \exp( \e \Gamma n)
$$ 
for some $\Delta>0$.
\end{cor}

As a matter of fact, \ref{StatisticThm} gives an elementary and unified proof of the following class of  groups with purely exponential growth, which were established by different methods:
\begin{enumerate}
\item
hyperbolic groups (Coornaert \cite{Coor}); 
\item
groups acting on CAT($-1$) space, with finite Bowen--Margulis--Sullivan measure (Roblin \cite{Roblin});
\item
fundamental groups of compact rank-1 manifolds (Kneiper \cite{Kneiper2});
\item
relatively hyperbolic groups, with the word metric \cite{YANG7} and the  hyperbolic metric \cite{YANG8};
\item
mapping class groups, with the Teichm\"{u}ller metric (Athreya  et al. \cite{ABEM}).
\end{enumerate}
 
Let us  comment   on the last of these.   Our general methods work for mapping class groups as well as further other applications, with the only assumption that the action of mapping class groups   on Teichm\"{u}ller spaces (with Teichm\"{u}ller metric) is SCC with contracting elements.  On the other hand, we need point out that computation of the precise value of the critical exponent, the entropy of Teichm\"{u}ller geodesic flows  which is $6g-6$, is out of scope of this approach.

In addition to recovering the above well-known results,    the following new classes of groups with purely exponential growth have been established as direct consequences of Theorem \ref{StatisticThm}: 

\begin{thm}\label{PEGrowthExamp}
The following class of groups have purely exponential growth:
\begin{enumerate}
\item
The action on their Cayley graphs of  $\mathrm{Gr}'(1/6)$-labeled graphical small cancellation groups   with finite components  labeled by a finite set $S$;
\item
CAT(0) groups with rank-1 elements;
\item  
the action on the Salvetti complex of right-angled Artin groups that are not direct products;
\item
the action on the Davis complex of a right-angled Coxeter group that is not virtually a direct product of non-trivial groups.
\end{enumerate}
\end{thm}

The items (3) and (4) are two important classes of CAT(0) groups with rank-1 elements. A detailed proof   is  given in the subSection \ref{ProofPEGrowthExamp}.

\subsection{Tools: Growth-tightness and Extension Lemma}
As the title indicates, we are now going to describe   two basic tools  throughout this study. The first is a growth-tightness theorem   adressing     the question that which \textit{subsets} are growth-tight.   

Let us first give some historical background on the notion of growth-tightness. It was introduced by Grigorchuk and de la Harpe in \cite{GriH} for a \textit{group}:  roughly speaking, a group is called growth-tight if the growth rate strictly decreases under taking quotients.  Their main motivation is perhaps that if,  {for some generating set, the growth rates of a growth-tight group achieve an infimum, called the \textit{entropy} (the entropy realization problem), then the group is Hopfian. In practice, it appears to be quite difficult to solve the entropy realization problem, whereas the Hopfian nature of a group is relatively easy to  establish (for instance as a consequence of residual  finiteness).} So conversely, Sambusetti \cite{Sam2} constructed the first examples of groups with unrealized entropy: indeed, he showed the growth-tightness of a free product of any two  groups if it is not $\cong \mathbb Z_2 \star \mathbb Z_2$, so his examples are free products of non-Hopfian groups.
 
Since their introduction, the property of growth-tightness was then established by Arzhantseva and Lysenok \cite{AL} for hyperbolic groups, by Sambusetti \cite{Sam2}  \cite{Sam3}  for free products and cocompact Kleinian groups, and by Dal'bo, Peign\'e,   Picaud and Sambusetti \cite{DPPS} for geometrically  finite Kleinian groups with parabolic gap property. The present author \cite{YANG6} realized their most arguments  in a broad setting and showed growth-tightness of groups with non-trivial Floyd boundary, subsequently generalizing the result to any group acting properly and cocompactly on a geodesic metric space with a contracting element. This result  was   achieved   independently and simultaneously  by Arzhantseva et. al \cite{ACTao}, who also proved the very interesting result that  the action of mapping class groups on Teichm\"{u}ller space is SCC (in our terminology) and   is growth-tight in the   sense of Grigorchuk and de la Harpe. 

The next main theorem of this study provides a class of growth-tight sets called \textit{barrier-free elements}. With a basepoint $o$ fixed, an element $h\in G$ is called \textit{$(\epsilon, M, g)$-barrier-free} if there exists an \textit{$(\epsilon, g)$-barrier-free} geodesic $\gamma$ with $\gamma_-\in B(o, M)$ and $\gamma_+\in B(ho, M)$:  there exists no $t\in G$ such that $d(t\cdot o, \gamma), d(t\cdot go, \gamma)\le \epsilon$. We refer to Definition \ref{barriers} for more details. 

\begin{mainthm}[Growth-tightness] \label{GrowthTightThm}
Suppose that $G$ has a  SCC action on a geodesic space $(\mathrm Y, d)$ with a contracting element. Then there exist  constants $\epsilon, M>0$ such that for any given $g \in G$, we have
$$
\e {\mathcal V_{\epsilon, M, g}} < \e G
$$
where $\mathcal V_{\epsilon, M, g}$ denotes the set of $(\epsilon, M, g)$-barrier-free elements.   
\end{mainthm}

\begin{rem}(about the constants $\epsilon, M$)
Any constant $M$ satisfying Definition \ref{StatConvex} works here. By Lemma \ref{concaveRegion}, the constant $M$ can be chosen as large as possible in applications. The constant $\epsilon=\epsilon(M, \mathbb F)$ depends on the choice of a contracting system $\mathbb F$ (Convention \ref{ConvExtensionLemma}) by Lemma \ref{injective}. The bigger the constant $\epsilon>0$, the smaller   the set $\mathcal V_{\epsilon, M, g}$ is.
\end{rem}

It is easy to see that this result generalizes all the above-mentioned results for growth-tightness of groups: 
\begin{cor}[=\ref{growthtightcor}]
Under the same assumption as in \ref{GrowthTightThm}, we have 
$$
\e {\bar G} < \e G,
$$
for any quotient $\bar G$ of $G$ by an infinite normal subgroup $N$, where $\e {\bar G}$ is computed with respect to the proper action of $\bar G$ on $\mathrm Y/N$ endowed with the quotient metric defined by $\bar d(Nx, Ny):=d(Nx, Ny)$.
\end{cor}



We shall consider two applications to the growth-tightness of a weakly quasi-convex subgroup. A subset $X$ is called  \textit{weakly $M$-quasi-convex} for a constant $M>0$ if for any two points $x,y$ in $X$, there exists a geodesic $\gamma$ between $x$ and $y$ such that $\gamma \subset N_M(X)$. A subgroup $\Gamma$ of $G$ is called \textit{weakly quasi-convex} if  the set $\Gamma\cdot o$ is  weakly quasi-convex   for some point $o \in \mathrm Y$. A sample application of \ref{GrowthTightThm} establishes the following.

\begin{thm}[=\ref{wqcGTight}] \label{weakqconvexity}
Let $G$ admit a SCC action on a geodesic   space $(\mathrm Y, d)$ with a contracting element. Then any weakly quasiconvex subgroup $\Gamma$ of infinite index  is growth-tight.
\end{thm}
\begin{rem} [on the proof]
Roughly speaking, the  strategy when utilizing \ref{GrowthTightThm} is to find interesting sets with certain negative  or non-negative curvatures. These sets will be ``barrier-free'', thanks to the exclusivity of negative curvature and non-negative curvature.  We refer the reader to Theorem \ref{wqcGTight} and Lemma \ref{SCCBasepoint}. This idea turns out to be very fruitful, which we shall pursue in a subsequent paper \cite{YANG11}.
\end{rem}

The first corollary considers the class of convex-cocompact subgroups in $\mathbb {Mod}$. Free convex-cocompact subgroups exist in abundance (\cite[Theorem 1.4]{FarbMosher}), but one-ended ones are unknown at present.  Although we cannot determine whether a given mapping class group has any large free convex-cocompact subgroups (cf. \ref{LargeFreeThm}), the \textit{growth-tightness} part is indeed true:
 
\begin{cor}[=  \ref{FMtight}]
Any convex-cocompact subgroup $\Gamma$ in  $G\in \mathbb {Mod}$ is growth-tight: $\e \Gamma < \e G.$  
\end{cor}

The next corollary applies to the class of \textit{cubulated} groups which acts properly on a non-positively curved cubical complex. Recall that a subgroup is \textit{cubically convex} if it acts cocompactly on a convex subcomplex. The following therefore answers positively \cite[Problem 9.7]{FutWise}.

\begin{cor}[=\ref{cubconvex}]
Suppose that a group   $G$ acts properly and   cocompactly on a $\mathrm{CAT}(0)$ cube complex $\mathrm Y$ such that $\mathrm Y$ does not decompose as a product. Then any weakly quasi-convex subgroup of infinite index in $G$ is growth-tight. In particular, any  cubically convex subgroup   is growth-tight if it is of infinite index.
\end{cor}

The second tool is a so-called \textit{extension lemma}, which is a simple but quite useful consequence of the existence of a contracting element. To illustrate this, we emphasize that proofs of \ref{LargeFreeThm}, \ref{StatisticThm}, and \ref{GrowthTightThm} are constructed via repeated applications of extension lemmas.

For convenience, we state here   a simplified version and refer the reader to \textsection \ref{extensionlemSec} for other versions.
\begin{lem}[Extension Lemma] 
There exist $\epsilon_0>0$ and a set $F$ of three elements in $G$ with the following property. For any two $g, h \in G$, there exists $f \in F$ such that $g\cdot f\cdot h$ is almost a geodesic: $$|d(o, g  f  h \cdot o) - d(o, go)-d(o, ho)| \le \epsilon_0.$$
\end{lem}
\begin{rem}
This result is best illustrated for free groups with standard generating sets. In this case, we choose $F=\{a, b, a^{-1}\}$ and $\epsilon_0=1$. To the best of our knowledge,  this result was first proved by Arzhantseva and Lysenok \cite[Lemma 3]{AL} for hyperbolic groups, and reproved later in \cite[Lemma 4.4]{Gran} and \cite[Lemma 2.4]{Gouezel}. In  joint work with  Potyagailo \cite{PYANG}, we have proved a version for relatively hyperbolic groups. The proof generalizes to the current setting with more advanced versions. 
\end{rem}

To finish this introductory section, we compare with the study of acylindrically hyperbolic groups   formulated in \cite{Osin6}. By a result of Sisto \cite{Sisto} (see also \cite[Appendix]{YANG6}),  the existence of a contracting element produces a hyperbolically embedded subgroup in the sense of \cite{DGO}. Thus, they  all belong to the category of  acylindrically hyperbolic groups,  which are studied previously in various guises in   \cite{BF2}, \cite{BBF}, \cite{Hamen}, and \cite{DGO}  and  in a continually growing body of literature.  The emphasis of our study is, however, on understanding the asymptotic geometry of these groups,    relying on their concrete actions rather than on the actions as tools.   
 
Furthermore, the reader should distinguish our definition of a contracting element from others in the literature  (e.g., \cite{Beh2}, \cite{Sisto}): the projection in definition is meant to be a \textbf{nearest point} projection, whereas some authors   take a more flexible one. As a consequence, a contracting element in their sense is a quasi-isometric invariant, whereas this is not the case for our definition, as shown by a recent example in \cite{ACGH}. However, it is this definition which brings the extension lemma into play.  
\\ 

\paragraph{\textbf{The structure of  this paper}} As a prerequisite, \textsection \ref{Section2} introduces the  extension lemma and gives two immediate applications to the positive density of contracting elements (cf. Proposition \ref{Positivedensity}), the finite depth of dead-ends (cf. Proposition \ref{deadend}). The next Sections \ref{Section3},   \ref{Section4}, and \ref{Section5}, which  are mutually independent, contains respectively the proofs of  \ref{LargeFreeThm},  \ref{GrowthTightThm} and \ref{StatisticThm}.    In the final  \textsection\ref{Section6}, we give a way to construct non-convex-compact SCC actions.

This, the first of a series of papers, lays the foundation for the further study of groups with contracting elements. In a subsequent paper \cite{YANG11}, we make use of the results established here to investigate the genericity of contracting elements in groups ($\mathbb {Hyp}$ - $\mathbb {Mod}$).

\ack
The author is grateful  to Jason Behrstock for pointing out Theorem \ref{Behrstock},   Laura Ciobanu,   Thomas Koberda,  Leonid Potyagailo, Mahan Mj and Dani Wise for helpful conversations.  Thanks also go to Lewis Bowen for pointing out an error in the argument in proving  (stronger) \ref{LargeFreeThm}.  Ilya  Gehktman gave him valuable feedback  and   shared with him lots of  knowledge in  dynamics of Teichm\"{u}ller geometry.

\section{An extension lemma}\label{Section2}
This section introduces the basic tool: an extension lemma,  in which a notion of a contracting element plays an important role. We first fix some notations and conventions.

\subsection{Notations and conventions}\label{ConvSection}
Let $(\mathrm Y, d)$ be a proper geodesic metric space. Given a point $y \in Y$ and a subset $X \subset \mathrm Y$,
let $\pi_X(y)$ be the set of points $x$ in $X$ such that $d(y, x)
= d(y, X)$. The \textit{projection} of a subset
$A \subset \mathrm Y$ to $X$ is then $\pi_X(A): = \cup_{a \in A} \pi_X(a)$.

Denote $\proj_X(Z_1, Z_2):=\diam{\pi_X({Z_1\cup Z_2})}$, which is the diameter of the projection of the union $Z_1\cup Z_2$ to $X$. So $d_X^\pi(\cdot, \cdot)$ satisfies the triangle inequality
$$
d_X^\pi(A, C) \le d_X^\pi(A, B) +d_X^\pi(B, C).
$$ 


We always consider a rectifiable path $\alpha$ in $\mathrm Y$ with arc-length parameterization.  Denote by $\len (\alpha)$ the length
of $\alpha$, and by $\alpha_-$, $\alpha_+$ the initial and terminal points of $\alpha$ respectively.   Let $x, y \in \alpha$ be two points which are given by parameterization. Then $[x,y]_\alpha$ denotes the parameterized
subpath of $\alpha$ going from $x$ to $y$. We also denote by $[x, y]$ a choice of a geodesic in $\mathrm Y$ between $x, y\in\mathrm Y$.  
\\
\paragraph{\textbf{Entry and exit points}} Given a property (P), a point $z$ on $\alpha$ is called
the \textit{entry point} satisfying (P) if $\len([\alpha_-, z]_\alpha)$ is
minimal   among the points
$z$ on $\alpha$ with the property (P). The \textit{exit point} satisfying (P) is defined similarly so that $\len([w,\alpha_+]_\alpha)$ is minimal.

A path $\alpha$ is called a \textit{$c$-quasi-geodesic} for $c\ge 1$ if the following holds 
$$\len(\beta)\le c \cdot d(\beta_-, \beta_+)+c$$
for any rectifiable subpath $\beta$ of $\alpha$.

Let $\alpha, \beta$ be two paths in $\mathrm Y$. Denote by $\alpha\cdot \beta$ (or simply $\alpha\beta$) the concatenated path provided that $\alpha_+ =
\beta_-$.

Let $f, g$ be real-valued functions with domain understood in
the context. Then $f \prec_{c_i} g$ means that
there is a constant $C >0$ depending on parameters $c_i$ such that
$f < Cg$.  The symbols    $\succ_{c_i}  $ and $\asymp_{c_i}$ are defined analogously. For simplicity, we shall
omit $c_i$ if they are   universal constants.
\subsection{Contracting property}
\begin{defn}[Contracting subset]\label{ContrDefn}
Let $\mathcal {QG}$ denote a preferred collection of quasi-geodesics in
$\mathrm Y$. For given $C\ge 1$, a subset $X$ in $\mathrm Y$ is called $C$-\textit{contracting} with respect to $\mathcal {QG}$ if for any quasi-geodesic $\gamma \in \mathcal {QG}$ with $d(\gamma,
X) \ge C$, we have
$$\proj_{X} (\gamma)  \le C.$$
 A
collection of $C$-contracting subsets is referred to
as a $C$-\textit{contracting system} (with respect to
$\mathcal {QG}$).
\end{defn}
 
\begin{example} \label{examples} We note the following examples in various contexts.
\begin{enumerate}
\item
Quasi-geodesics and quasi-convex subsets are contracting with respect
to the set of all quasi-geodesics in hyperbolic spaces. 
\item
Fully quasi-convex subgroups (and in particular, maximal parabolic
subgroups) are contracting with respect to the set of all
quasi-geodesics in relatively hyperbolic groups (see Proposition
8.2.4 in \cite{GePo4}).
\item
The subgroup generated by a hyperbolic element is contracting  with
respect to the set of all quasi-geodesics in groups with non-trivial
Floyd boundary.  This is described in \cite[Section 7]{YANG6}.
\item
Contracting segments in CAT(0)-spaces in the sense of in
 Bestvina and  Fujiwara are contracting here with respect to the set of
geodesics (see Corollary 3.4 in \cite{BF2}).
\item
The axis of any pseudo-Anosov element is contracting relative to geodesics by   Minsky \cite{Minsky}.
\item
Any finite neighborhood of a contracting subset is still contracting
with respect to the same $\mathcal {QG}$.
\end{enumerate}
\end{example}

\begin{conv}\label{ContractingConvention}
In view of Examples \ref{examples}, the preferred collection
$\mathcal {QG}$ in the sequel will always be the set of all
geodesics in $\mathrm Y$.   
\end{conv}


In fact, the
contracting notion is equivalent to the following one considered by  Minsky \cite{Minsky}. A proof  given in \cite[Corollary 3.4]{BF2} for CAT(0) spaces is valid in  the general case. Despite this equivalence, we always work with the above definition of the contracting property. 

\begin{lem}\label{CharaContraction}
A subset $X$ is  contracting in $\mathrm Y$ if and only if any open ball $B$ missing $X$ has a uniformly bounded projection to $X$.   
\end{lem}

We collect a few properties that will be used often later on. The proof   is straightforward applications of contracting property, and is left to the interested reader.
\begin{prop}\label{Contractions}
Let $X$ be a contracting set.
\begin{enumerate}
\item
(Quasi-convexity) \label{qconvexity}  $X$ is \textit{$\sigma$-quasi-convex} for a function $\sigma: \mathbb R_+ \to \mathbb R_+$: given $c \ge
1$,  any $c$-quasi-geodesic with endpoints in $X$ lies in the
neighborhood $N_{\sigma(c)}(X)$. 
\item
(Finite neighborhood) \label{nbhd}  Let $Z$ be a set with finite Hausdorff distance to $X$. Then $Z$ is contracting.

There exists $C>0$ such that the following hold:

\item 
\label{Equivalence}
For any   geodesic segment $\gamma$, the following holds
$$|\proj_X(\{\gamma_-, \gamma_+\}) - \proj_X(\gamma)|\le C.$$

\item
(1-Lipschitz projection) \label{1Lipschitz} $\proj_X(\{y, z\})\le d(y, z)+C$.

\item
(Projection point) \label{ProjPoint} Let $\gamma$ be a geodesic segment such that $\gamma_-\in X$, and  $x\in X$ be a projection point of $\gamma_+$ to $X$. Then $d(x, \gamma) \le C$.
\item
(Coarse projections) \label{CoarseProj} For any two points $x\in X, y \notin X$, we have $$|d(x, \pi_X(y)) - \big(d(x, y)-d(y, X)\big)|\le C.$$  
\end{enumerate}
\end{prop}






In most cases, we are interested in a  contracting system with a \textit{$\mathcal R$-bounded intersection} property for a function $\mathcal R: \mathbb R_{\ge 0}\to \mathbb R_{\ge 0}$ if the
following holds
$$\forall X\ne X'\in \mathbb X: \;\diam{N_r (X) \cap N_r (X')} \le \mathcal R(r)$$
for any $r \geq 0$. This property is, in fact, equivalent to a \textit{bounded intersection
property} of $\mathbb X$:  there exists a constant $B>0$ such that the
following holds
$$\proj_{X'}(X) \le B$$
for $X\ne X' \in \mathbb X$. See \cite{YANG6} for further discussions.
 
\begin{rem}
Typical examples include sufficiently separated quasi-convex subsets
in hyperbolic spaces, and parabolic cosets in relatively hyperbolic
groups (see \cite{DruSapir}).
\end{rem}

\subsection{Admissible paths}
The notion of an admissible path is defined relative to   a  contracting system $\mathbb X$ in $\mathrm Y$. Roughly speaking, an admissible path can be thought of as a
concatenation of quasi-geodesics which travels alternatively near
contracting subsets and leave them in an orthogonal way. 


\begin{defn}[Admissible Path]\label{AdmDef}
Given $D, \tau \ge 0$ and a function $\mathcal R: \mathbb R_{\ge 0} \to \mathbb R_{\ge 0}$, a path $\gamma$  is called \textit{$(D,
\tau)$-admissible} in  $\mathrm Y$, if  the path $\gamma$ contains a sequence of disjoint geodesic subpaths $p_i$ $(0\le i\le n)$ in this order, each associated to a contracting subset $X_i \in \mathbb X$, with the following   called \textit{Long Local} and \textit{Bounded Projection} properties:
\begin{enumerate}

\item[\textbf{(LL1)}]
Each $p_i$ has length bigger than  $D$, except that  $(p_i)_- =\gamma_-$ or $(p_i)_+=\gamma_+$,

\item[\textbf{(BP)}]
For each $X_i$,  we have 
$$
\proj_{X_i}((p_{i})_+,(p_{i+1})_-)\le \tau
$$
and 
$$
\proj_{X_i}((p_{i-1})_+, (p_{i})_-)\le \tau
$$
when $(p_{-1})_+:=\gamma_-$ and $(p_{n+1})_-:=\gamma_+$ by convention.

\item[\textbf{(LL2)}]
Either $X_i \ne X_{i+1}$ has $\mathcal R$-bounded intersection or 
$d((p_i)_+, (p_{i+1})_-)>D$,


\end{enumerate}
\paragraph{\textbf{Saturation}} The collection of $X_i \in \mathbb X$ indexed as above, denoted by $\mathbb X(\gamma)$, will be referred to as contracting subsets for $\gamma$. The union of all $X_i \in \mathbb X(\gamma)$ is called the \textit{saturation} of $\gamma$. 
\end{defn}

The set of endpoints of $p_i$ shall be refered to as the \textit{vertex set} of $\gamma$. We call $(p_{i})_-$ and $(p_{i})_+$ the corresponding \textit{entry vertex} and \textit{exit vertex}   of $\gamma$ in $X_i$. (compare with entry and exit points in subSection \ref{ConvSection}) 

\begin{rem}
In \cite{YANG6}, an admissible path is defined so that the subpath between  $p_{i}$ and $p_{i+1}$ is a geodesic and $\proj_{X_i}[(p_{i+1})_-, (p_{i})_+]\le \tau$. By Proposition \ref{Contractions}.\ref{Equivalence},  this condition is equivalent to ($\mathbf{BP}$). Up to replace them by corresponding geodesics, we obtain a notion of admissible path originally defined in \cite{YANG6}. We allow a non-geodesic path so  it is easier to verify \textbf{(BP)}.  
\end{rem}

\begin{rem}[Bounded intersection]
In most applications, the contracting system relative to which we consider admissible paths has bounded intersection. Hence, it suffices to show that $X_i$ and $X_{i+1}$ are distinct in the verification of Condition (\textbf{LL2}). 
\end{rem}


By definition, a sequence of points $x_i$ in a path $\alpha$   is called \textit{linearly ordered} if $x_{i+1}\in [x_i, \alpha_+]_\alpha$ for each $i$. 

\begin{defn}[Fellow travel]\label{Fellow}
Assume that $\gamma = p_0 q_1 p_1 \cdots q_n p_n$ is a $(D, \tau)$-admissible
path, where each $p_i$ has two endpoints in $X_i \in \mathbb X$. The paths $p_0, p_n$ could be trivial. 

Let
$\alpha$ be a path such that $\alpha_- = \gamma_-,
\alpha_+=\gamma_+$. Given $\epsilon >0$, the path $\alpha$ \textit{$\epsilon$-fellow travels} $\gamma$ if there exists a sequence of linearly ordered points $z_i,
w_i$ ($0 \le i \le n$) on $\alpha$ such that  
$d(z_i, (p_{i})_-) \le \epsilon, \;d(w_i, (p_{i})_+) \le \epsilon.$
\end{defn}

The basic fact  is that a ``long" admissible path is a quasi-geodesic.
 
\begin{prop}\label{admissible} 
Let $C$ be the contraction constant of $\mathbb X$. For any $\tau>0$, there are constants $B=B(\tau), D=D(\tau), \epsilon = \epsilon(\tau), c = c(\tau)>0$ such that the following
holds.

Let $\gamma$ be a $(D, \tau)$-admissible path and $\alpha$ a geodesic  between $\gamma_-$ and
$\gamma_+$. Then
\begin{enumerate}
\item
For    a contracting subset $X_i \in \mathbb X(\gamma)$ with $0 \le i \le
n$,
$$\proj_{X_i}(\beta_1) \le B,\; \proj_{X_i}(\beta_2)  \le B$$ where $\beta_1 =[\gamma_-, (p_i)_-]_\gamma, \beta_2 =[(p_i)_+, \gamma_+]_\gamma$. 
\item
$\alpha \cap N_C(X) \ne \emptyset$ for every $X\in \mathbb X(\gamma)$.
\item
$\alpha$  $\epsilon$-fellow travels $\gamma$. In particular,  $\gamma$ is a $c$-quasi-geodesic.
\end{enumerate}
\end{prop}
\begin{proof}[Sketch of the proof]
The content of this proposition was proved in \cite[Proposition 3.3]{YANG6}. The constant $D>0$ is taken to be sufficently large but independent of $n$, and the first statement was proved by induction on $n$ as \cite[Corollary 3.7]{YANG6}. Assuming Assertion (1), the second and third statements follow as consequences. For instance, if $D>2B+C$, then we must have $\alpha \cap N_C(X) \ne \emptyset$ by the contracting property. Moreover, we can set $\epsilon:=2C+B$.  We refer the interested reader to \cite[Section 3]{YANG6} for more details. 
\end{proof}

The next result generalizes \cite[Lemma 4.4]{YANG6} by a similar proof. The main use of this lemma (the second statement) is to construct the following type of paths in verifying that an element is contracting.

\begin{defn}\label{uniformadmissible}
Let $L, \Delta>0$. With notations in definition of a $(D, \tau)$-admissible path $\gamma$, if the following holds $$|d((p_{i+1})_-, (p_{i})_+)-L|\le \Delta$$ for each $i$, we say that $\gamma$ is a \textit{$(D, \tau, L, \Delta)$-admissible} path.  
\end{defn}

\begin{prop}\label{saturation}
Assume that $\mathbb X$ has bounded intersection in   admissible paths considered in the following statements.
For any $\tau>0$ there exists $D=D(\tau)>0$ with the following properties.
\begin{enumerate}
\item
For any $L,\Delta>0$, there exists $C=C(L, \Delta)>0$ such that the saturation of a $(D, \tau, L, \Delta)$-admissible path is  $C$-contracting.  
\item
For any $L,\Delta, K>0$, there exists $C=C(L, \Delta, K)>0$  such that if the entry and exit vertices of a $(D, \tau, L, \Delta)$-admissible path $\gamma$ in each $X\in \mathbb X(\gamma)$ has distance bounded above by $K$, then $\gamma$ is $C$-contracting.
\end{enumerate}
\end{prop}

\begin{proof}
Consider a $(D, \tau, L, \Delta)$-admissible path $\gamma=p_0q_1p_1\cdots
p_{n-1} q_n$, where $p_i$ are geodesics with two endpoints in $X_i\in \mathbb X(\gamma)$ and 
$q_i$ are geodesics. Denote by $C_0$ the contraction constant for all $X_i$. For $\tau>0$, let $D=D(\tau), B=B(\tau)$ provided by Proposition \ref{admissible}.  

\textbf{(1).} Denote the saturation  $A:=\bigcup_{X_i\in \mathbb X(\gamma)} X_i$. Consider a geodesic $\alpha$ such that $$\alpha \cap
N_{C}(A) =\emptyset$$ where the constant $C>C_0$ (given below in (\ref{CValue})) depend $L,\Delta$ and $C_0$. Let $z \in Z, w \in W$ be projection
points of $\alpha_-, \alpha_+$ respectively to $A$, where $Z, W\in \mathbb X(\gamma)$ appear on $\gamma$ in this order. Without loss of generality,
assume that $d(z,w)= {\proj_A(\alpha)}$. The purpose of the proof is to prove $d(z,w)\le C$. 

If $Z=W$, the conclusion then follows from the contracting property of $Z$. So, assume that $Z\ne W$ below. Since $\mathbb X$ has bounded projection, there exists a constant, for simplicity, the same $B>0$ from Proposition \ref{admissible} such that 
\begin{equation}\label{bddprojEQ}
\proj_{X}(X') \le B
\end{equation}
for $X\ne X'\in \mathbb X$. 
 
We first observe that 
\begin{equation}\label{bddintEQ}
\max\{\diam{[z,\alpha_-]\cap
N_{C_0}(X)}, \;\diam{[w,\alpha_+]\cap
N_{C_0}(X)}\} \le C_0
\end{equation} for any $X\in \mathbb X(\gamma)$. Indeed, consider the entry point of $[z,\alpha_-]$ in $N_{C_0}(X)$ which is $C_0$-lcose to $X \subset A$, so the exit point must be within a $C_0$-distance to the entry point, since $z$ is a shortest point in $A$ to $\alpha_-$.  The (\ref{bddintEQ}) thus follows.

Let $z'\in Z$ and $w'\in W$ be the corresponding exit and entry \textbf{vertices} (cf. Definition \ref{AdmDef}) of $\gamma$ in $Z$ and $W$.  

Applying (\ref{bddintEQ}) to $X=W$, we see $\proj_W([w,\alpha_+])\le 5C_0$ by the contracting property. 
By assumption, $\alpha \cap N_{C_0}(Z)=\emptyset$ and then ${\proj_{Z}(\alpha)}\le C_0$.  Using (\ref{bddprojEQ}), we obtain: 
$$
\begin{array}{lll}
d(z, z') &\le  {\proj_{Z}(\alpha)} +
 {\proj_{Z}([w,\alpha_+])}+ \proj_{Z}(W) +  {\proj_{Z}([z', w']_\gamma)}\\
 &\le 6C_0 + 2B,
 \end{array} 
$$
where   $\proj_{Z}([z', w']_\gamma)\le B$ follows by Proposition \ref{admissible}. 
Proceeding similarly, we obtain that $$d(w, w')\le 6C_0 + 2B.$$ 

In order to bound the distance $d(z, w)$, it remains to prove that the subpath $[z', w']_\gamma$ between $z', w'$ in $\gamma$ contains no other vertices than $z', w'$.  Indeed, by definition of $\gamma$ being a $(D, \tau, L,\Delta)$-admissible path, we have $|d(z', w')-L|\le \Delta$. Consequently, we obtain
\begin{equation}\label{CValue}
d(z, w) \le d(z', w')+d(z,z')+d(w,w')\le   C:=4(3C_0 + B)+L+\Delta,
\end{equation} 
concluding the proof of the assertion (1). 

Assume, by way of contradiction, that $[z', w']_\gamma$ contains a geodesic segment $p$ associated to a contracting set $X\in \mathbb X(\gamma)$. By definition, $\len(p)>D.$ It is obvious that $Z\ne X\ne W$.  

On the other hand, noting that $\alpha$ lies outside $N_\epsilon(A)$ and $X\subset A$, we have $\alpha\cap N_{C_0}(X)=\emptyset$ for $\epsilon>C_0$ so $\proj_{X}(\alpha)\le C_0$. By (\ref{bddintEQ}),  we see ${\proj_{X}([z,\alpha_-])}, \proj_X([w,\alpha_+])\le 5C_0$. Thus the length of $p$ gets bounded as follows:
$$
\begin{array}{lll}
d(p_-, p_+)&\le {\proj_{X}([p_-, z']_\gamma)} + {\proj_{X}(Z)} +
 {\proj_{X}([z,\alpha_-])}\\
 & \;+ \proj_{X}(\alpha) +  {\proj_{X}([\alpha_+, w])}+{\proj_{X}(W)}+{\proj_{X}([w', z]_\gamma)}\\
 &\le 4B+11C_0,
\end{array} 
$$
where $\proj_{X}([p_-, z']_\gamma, {\proj_{X}([w', z]_\gamma)} \le B$ by Proposition \ref{admissible} (1), and ${\proj_{X}(Z)}, {\proj_{X}(W)} \le B$ by (\ref{bddprojEQ}).  Consequently, this gives a contradiction, by further setting 
\begin{equation}\label{DValue}
D>4B+11C_0,
\end{equation} 
whence the proof of (1) is complete.

\textbf{(2).} We follow the same line as \textbf{(1)}: the notation $A$ now denotes the set of vertices of $\gamma$. Note also that $z$ and $w$ are vertices of $\gamma$. The main difference comes from treating the case that $Z=W$:  the contracting property now follows by the assumption that $d(z, w)$ is uniformly bounded  by $K$. So the contraction constant $C$   depends on $K$ as well. We leave the details to the interested reader. 
\end{proof}

Before passing to further discussions, let us  introduce a few ways to manipulate existing admissible paths to produce new ones. 

\textbf{Subpath:} let $\gamma$ be a $(D, \tau)$-admissible path. An \textit{admissible subpath} $\beta$ is a subpath between two vertices in $\gamma$. It is clear that, an admissible subpath is $(D, \tau)$-admissible.
 
\textbf{Concatenation:} let $\alpha, \beta$ be two $(D, \tau)$-admissible paths. Suppose the last contracting subset associated to the geodesic segment $p$ of $\beta$ is the same as the first contracting subset  to the geodesic segment  $q$ of $\gamma$. A new path $\gamma$ can be thus formed by concatenating paths $$\gamma:=[\alpha_-, p_-]_{\alpha}   [p_-, q_+] [q_+, \beta_+]_{\beta}.$$ 
A useful observation is that if $d(p_-, q_+)>D$, then $\gamma$ is $(D, \tau)$-admissible. 
\\

\paragraph{\textbf{Path label convention}}
We conclude this subsection with the following terminology, which is designed to be consistent with the common one -- paths labeling by words in  Cayley graphs. 

Let $\{g_i\in G: 1\le i\le n; \; n\le \infty\}$ be a (possibly  infinite) sequence of elements. Fixing a basepoint $o\in X$, we plot a sequence of  points $x_i=h_i\cdot o$ $(i=0, \cdots n)$ where $h_i=g_0\cdots g_{i-1}$ with $g_0=1$, and then connect $x_i, x_{i+1}$ by a geodesic to define a piecewise geodesic path $\gamma$.  We call such a path $\gamma$ \textit{labeled by}  the sequence $\{g_i\}$, and $x_i$ the \textit{vertices} of $\gamma$. If $n$ is finite, the path $\gamma$ is also said to be \textit{labeled} by the (product form of) element $g_1g_2\cdots g_n$, it \textit{represents} the element $g_1 g_2\cdots g_n$ in $G$. .

Very often, we need to write the path $\gamma$ explicitly as follows
$$
\gamma:=\path{g_1}\path{g_2}\cdots \path{g_{n-1}}\path{g_n}\cdots
$$
where $\path{g_i}$ denotes a choice of a geodesic between $x_i, x_{i+1}$.

By abuse of language, we say that any translate of a path $\gamma$ by an element in $G$ is also \textit{labeled} by $\{g_i\}$.  

By definition, a  \textit{labeled} path by a bi-infinite sequence of elements is defined   as  the union of two labeled paths by  one-sided infinite sequence of elements.  

\subsection{Contracting subgroups}\label{SScontracting}
We first setup a few definitions. An infinite subgroup $H$ in $G$ is
called \textit{contracting} if for some (hence any by Proposition \ref{Contractions}.\ref{nbhd}) $o \in \mathrm Y$, the
subset $Ho$ is contracting in $\mathrm Y$. In fact, we usually deal with a contracting subgroup $H$ with \textit{bounded intersection}: the collection of subsets
$$\{gH\cdot o: g\in G\}$$ is a contracting system with bounded intersection
in $\mathrm Y$. (In \cite{YANG6}, a contracting subgroup $H$ with bounded intersection was called strongly contracting.)

\begin{lem}\label{normalizer}
Assume that $G$ acts properly on $(\mathrm Y, d)$.
If $H$ is a contracting subgroup, then $[N(H): H]< \infty$, where $N(H)$ is the normalizer of $H$ in $G$. In particular, $N(H)$ is contracting.
\end{lem}
\begin{proof}
Let $g\in N(H)$ so $gH=Hg$. It follows that $gH\cdot o\subset N_D(H\cdot o)$ for $D:=d(o, go)$. Let $C$ be the contraction constant of $Ho$, and also satisfy Proposition \ref{Contractions}.\ref{qconvexity} such that any geodesic with two endpoints in $gHo$ lies in $N_C(gHo)$. 

Since $gHo$ is  unbounded, choose a geodesic $\gamma$ of length $\ge 2D+C$ with two endpoints in $gHo\subset N_D(H\cdot o)$. By contracting property, $\gamma$ has to intersect $N_C(Ho)$: if not, we have $\proj_{Ho}(\gamma)\le C$ and thus $\len(\gamma)\le \proj_{Ho}(\gamma) +d(\gamma_-, Ho)+d(\gamma_+, Ho)\le C+2D$, a contradiction. Together with  $\gamma\subset N_C(gHo)$, we obtain $N_C(gHo)\cap N_C(Ho) \ne \emptyset:$  there exists $h \in H$ such that $hg\in N(o, 2C)$. The finiteness of the set $N(o, 2C)$ shows that $[N(H): H]< \infty$. The finite neighborhood of a contracting set is contracting by Proposition \ref{Contractions}.\ref{nbhd}, so $N(H)$ is a contracting subgroup as well. 
\end{proof}


An  element $h \in G$ is called  
\textit{contracting} if the subgroup $\langle h \rangle$ is contracting, and the orbital map
\begin{equation}\label{QIEmbed}
n\in \mathbb Z\to h^no \in \mathrm Y
\end{equation}
is a quasi-isometric embedding.  The set of contracting elements is preserved under conjugacy.

Given  a contracting subgroup $H$, define a group $E(H)$  as
follows:
\begin{equation}\label{Ehdefn}
E(H):=\{g \in G: \exists r >0, gHo \subset N_r(Ho) \;  \text{and} \;  Ho \subset N_r(gHo)  \}.
\end{equation}

For a contracting element $h$, the structure of  $E(h):=E(\langle h\rangle)$ could be made   precise as follows.

\begin{lem}\label{elementarygroup}
Assume that $G$ acts properly on $(\mathrm Y, d)$. For a contracting element $h$,  the following statements hold:
\begin{enumerate}
\item
$[E(h): \langle h \rangle] <\infty$, and $E(h)$ is a contracting subgroup with bounded intersection. 
\item
$
E(h)=\{g\in G: \exists n > 0,\; (gh^ng^{-1}=h^n)\; \lor\;  (gh^ng^{-1}=h^{-n})\}.
$
\end{enumerate}
\end{lem}
 
\begin{proof}
\textbf{(1)}. Since $n\in \mathbb Z\to h^no\in \mathrm Y$ is a quasi-isometric embedding, the path $\gamma$ obtained by connecting consecutative dots is a quasi-geodesic which is   contracting. Hence, for any $r\gg 0$, the following unbounded intersection
$$
\diam{N_r(\langle h\rangle o) \cap N_r(g\langle h\rangle o)} =\infty  
$$ 
implies that there exists a uniform constant $C$ such that 
$$
 \langle h\rangle o  \subset N_C(g\langle h\rangle o) \;\text{and}\; g\langle h\rangle o  \subset N_C(\langle h\rangle o) 
$$
yielding   $g \in E(h)$.  As a consequence, the constant $r$ can be made uniform in defintion of $E(h)$. So, the assertion ``$[E(h): \langle h \rangle] <\infty$''  follows by a similar argument as in Lemma \ref{normalizer}.  

Furthermore, if $g\notin E(h)$, then $gE(h)o$ and $E(h)o$ have bounded intersection. The proper action also implies the uniformity of bounded intersection for all $g\notin E(h)$. Thus, $E(h)$ is contracting with bounded intersection.

\textbf{(2)}.  The right-hand set is contained in $E(h)$ as a subgroup.  For given $g\in E(h)$, there exists some $r>0$ such that $$g\langle h\rangle o \subset  N_r(\langle h\rangle o).$$ Since $\langle h\rangle o \subset N_D(\langle h\rangle go)$ for $D:=d(o, go)$, we have $g\langle h\rangle o \subset N_{r+D}(\langle h\rangle go)$. There exists a sequence of distinct pairs of  integers $(n_i, m_i)$ such that $d(g h^{n_i}o, h^{m_i}go)\le r+D$. It follows that $h^{-n_i}g^{-1}h^{m_i}g \in N(o, r+D)$, which is a finite set by the proper action. So there exist two distinct  pairs $(n_i, m_i)$ and $(n_j, m_j)$  such that $h^{-n_i}g^{-1}h^{m_i}g=h^{-n_j}g^{-1}h^{m_j}g$ and so $gh^{n_i-n_j}g^{-1}=h^{m_i-m_j}$.  An induction argument then proves
$$
g^l h^{(n_i-n_j)^l} g^{-l} = h^{(m_i-m_j)^l}
$$
for any $l \in \mathbb N$. So if $|n_i-n_j|\ne |m_i-m_j|$, a straightforward calculation gives a contradiction since  the map $n\to h^no$ is a quasi-isometric embedding. Letting $n=n_i-n_j$, we obtain $gh^ng^{-1}=h^{\pm n}$. The description of $E(h)$ thus follows.
\end{proof}


In what follows, the contracting subset 
\begin{equation}\label{axisdefn}
\ax(h)=\{f \cdot o: f\in E(h)\}
\end{equation} shall be called the \textit{axis} of $h$. Two  contracting elements $h_1, h_2\in G$  are called \textit{independent} if the collection $\{g\cdot \ax(h_i): g\in G;\ i=1, 2\}$ is a contracting system with bounded intersection. 



The following result could be thought of as an analog of the well-known fact in $\isom(\mathbb H^2)$ that two hyperbolic isometries in a discrete group  have  either disjoint fixed points or  the same   fixed points.  

\begin{lem}\label{fixedpoints}
Assume that a  group $G$ acts properly on $(\mathrm Y, d)$. For two contracting elements $h_1, h_2$, either $\langle h_1\rangle o$ and $\langle h_2\rangle o$ have bounded intersection, or they have finite Hausdorff distance and $h_1\in E(h_2)$.  

In particular, if $G$ is non-elementary, then there are infinitely many pairwise independent contracting elements.
\end{lem}
\begin{proof}
Assume that $\langle h_1\rangle o$ and $\langle h_2\rangle o$ have unbounded intersection: there exists a constant $r>0$ such that 
$$\diam{N_r(\langle h_1\rangle o) \cap N_r(\langle h_2\rangle o)} =\infty.$$
There exists a sequence of distinct pairs of  integers $(n_i, m_i)$ such that $d( h_1^{m_i}o, h_2^{n_i}o)\le 2r$, so $h_1^{-m_i}h_2^{n_i} \in N(o, 2r)$. The set $N(o, 2r)$ is thus finite by properness of the action. Hence, there exist two distinct pairs $(n_i, m_i)$ and $(n_j, m_j)$ such that  $h_1^{n_i-n_j}=h_2^{m_i-m_j}$.  This means that $E(h_1)$ and $E(h_2)$ are \textit{commensurable}, i.e. they have  finite index subgroups which are isomorphic. So $\ax(h_1)$ and $\ax(h_2)$ have finite Hausdorff distance.   The lemma is proved.
\end{proof}

Finally, we record the following elementary well-known fact; a proof is given for completeness.  We refer the reader to  \cite{FMbook} for the classification of \textit{periodic}, \textit{reducible}, and \textit{pseudo-Anosov} elements.  

\begin{lem}\label{pAContr}
In mapping class groups, a contracting element coincides with a pseudo-Anosov element with respect to the Teichm\"{u}ller metric.
\end{lem}
\begin{proof}
A pseudo-Anosov element $g$ is contracting by Minsky's theorem \cite{Minsky} that the axis of $g$ is contracting.    A contracting element is of infinite order by definition. So suppose that $g$ is   reducible. By definition, some power $g^n$ of a reducible element fixes a finite set of simple closed curves and thus commutes with Dehn twists around these curves. Hence, $g^n$ is of infinite index in its centralizer. On the other hand, a contracting element has to be of finite index in its centralizer by Lemma \ref{normalizer}. So we got a contradiction, completing the proof of the lemma.   
\end{proof}



\subsection{An extension lemma}\label{extensionlemSec}

 In this subsection, we do not demand that the action of $G$ on $\mathrm Y$ is proper. The only requirement is the existence of   three pairwise independent contracting elements. 

\begin{lem}[Extension Lemma]\label{extend3}
Suppose that a group $G$ acts on a geodesic metric space $(\mathrm Y, d)$ with   contracting elements.  Consider a collection of subsets $$\mathbb F=\{g\cdot \ax(h_i):   g\in G \},$$ where $h_i$ $(1\le i\le 3)$ are three pairwise independent contracting elements.  Then there exist $\epsilon_0, \tau, D>0$ depending only on $\mathbb F$ with the following property.  

Let $F$ be a set consisting of a choice of an element $f_i$ from each $E( h_i)$   such that $d(o, f_io)>D$ $(0\le i\le 3)$.
\begin{enumerate}
\item
For any two elements $g, h \in G$, there exist an element $f \in F$ such that $gfh\ne 1$  and the path labeled by $gfh$ is a $(D, \tau)$-admissible path. 
\item
As a consequence, the following holds
\begin{equation}\label{epsilonclose}
\max\{d(go,\gamma) , d(gfo, \gamma)\} \le \epsilon_0,
\end{equation}
for any geodesic $\gamma:=[o, gfho]$. 	

\end{enumerate}
\end{lem}

\begin{rem}
\begin{enumerate}
\item
We emphasize an arbitrary choice $F$ of three elements would satsify the lemma. We   \textbf{do} allow the possibility of $g$ or $h$ to be the identity. 
\item
The group action is not assumed to be a proper action. Another important source of examples comes from the  acylindrical action in \cite{DGO}, \cite{Osin6}.
\end{enumerate}
\end{rem}

The following simplified version of the extension lemma explains its name.
\begin{cor}[Extension Lemma: simplified version]\label{extend3easy}
Under the same assumption as Lemma \ref{extend3}.  Then for any two $g, h \in G$, there exist $f \in F$ such that the above (\ref{epsilonclose}) holds.
\end{cor}

We are   going to prove a  more general version which could deal with any number of elements.

\begin{lem}[Extension Lemma: infinite version]\label{extendinfty}
Under the same assumption as Lemma \ref{extend3}, there exist $\epsilon_0, \tau, D, c>0$  with the following property.  

Consider  a (finite or infinite, or bi-infinite) sequence of elements $\{g_i\in G:  |i|\le  \infty\}$. For any $g_i, g_{i+1}$, there exists $f_i \in F$ such that the following holds,
\begin{enumerate}
\item
A path $\gamma$ labeled by $(\cdots, g_i, f_i, g_{i+1}, f_{i+1}, g_{i+2}, \cdots)$ is a $(D, \tau)$-admissible path, and is a $c$-quasi-geodesic. 
\item
The product of any finite sub-sequence of consecutive elements in $$(\cdots, g_i, f_i, g_{i+1}, f_{i+1}, \cdots)$$ is non-trivial in $G$. 
\item
Let $\alpha$ be a geodesic between two vertices in $\gamma$.   Then each vertex in $[\alpha_-, \alpha_+]_\gamma$ has  a distance at most $\epsilon_0$ to $\alpha$.
\item
Consider two finite sequences $\{g_i: 1\le i \le m\}, \{g'_j: 1\le j \le n\}$ with $f_i, f'_i$ provided by the statement (1) satisfying that $g_1f_1g_2\cdots f_{m-1}g_m o=g_1'f_1'g_2'\cdots f'_{n-1}g'_no$. If $g_1=g_1'$, then $f_1=f_1'$.
\end{enumerate}
\end{lem}

Assume that $\mathbb F$ is a $C$-contracting system with $\mathcal R$-bounded intersection for some  $C>0$ and $\mathcal R: \mathbb R_{\ge 0}\to \mathbb R_{\ge 0}$. For simplicity we denote $A_k=\ax(h_k)$ below. 
The proof replies on the following observation. 
\begin{lem}\label{SeparationClaim}
There exists a constant $\tau$ such that for any element $g\in G$, the following holds
\begin{equation}\label{bddprojIJ}
\min\{\proj_{A_1}([o,
go]), \proj_{A_2}([o, go])\} \le \tau.
\end{equation}
 
\end{lem}
\begin{proof} 
By the contracting property, any geodesic segment with two endpoints within $C$-distance   from $A_i$ lies in the $3C$-neighborhood of $A_i$ for $i=1, 2$. We set 
\begin{equation}\label{tauValue}
\tau=\mathcal R(3C) +2C+1.
\end{equation}

Suppose, by contradiction, that there exist  some $g\in G$   such that $$\max\{\proj_{A_1}(\alpha), \proj_{A_2}(\alpha)\} > \tau$$ where $\alpha:=[o,
go]$. Let  $z_1\in A_1$ and $z_2\in A_2$ be the corresponding projection points of $go$  such that $\min\{d(o, z_1), d(o, z_2)\} > \tau$. Consider the exit points  $w_1, w_2$ of $\alpha$ in $N_C(A_1)$ and $N_C(A_2)$ respectively.   By the contracting property, we obtain that $\max\{d(z_1, w_1), d(z_2, w_2)\} \le 2C$. By the choice of $\tau$ (\ref{tauValue}), we have $$\min\{d(o, w_1), d(o, w_2)\} \ge \mathcal R(3C)+1.$$

On the other hand, assume that $d(o, w_1) \le d(o, w_1)$ for concreteness. Since $o, w_2 \in N_C(A_2)$,  by the quasi-convexity of $A_2$ stated above, we see that $[o, w_2]_\gamma$ lies in the $3C$-neighborhood of $A_2$. 
This implies  $$o, w_1 \in N_{3C}(A_1)\cap N_{3C}(A_2).$$  Since $d(o, w_1) > \mathcal R(3C)$, we get a contradiction since the pair  $A_1, A_2$ has $\mathcal R$-bounded intersection. Thus,  (\ref{bddprojIJ}) is proved. 
\end{proof}

We are in a position to give a proof of Lemma \ref{extendinfty}.
\begin{proof}[Proof of Lemma \ref{extendinfty}]
The proof consists in proving the statement (1), from which the statements (2) and (3) follow by Proposition \ref{admissible} in a straightforward way.

Let $\tau$ given by Lemma \ref{SeparationClaim}. First of all, for any $g, h\in G$, there exists at least one $A \in \{A_k: 1\le k\le 3\}$ such that
\begin{equation}\label{bddprojIJ2}
\max\{\proj_{A}([o,
go]), \proj_{A}([o, ho])\} \le \tau.
\end{equation}
Indeed, by Lemma \ref{SeparationClaim}, each $g$ and $h$ has a bounded projection by $\tau$ to at least two sets from $\{A_k: 1\le k\le 3\}$. Thus, there is at least one $A$ in common such that  (\ref{bddprojIJ2}) holds. 

We now make the choice of the constants. For a constant $D=D(\tau)$ given by Proposition \ref{admissible}, choose one element $f_k$ from each $E(h_k)$ such that $d(o, f_k o) >D$. Denote $F=\{f_1, f_2, f_3\}$. 
\\
\paragraph{\textbf{Construction of admissible paths.}} Let $\{g_i \in G\}$ be a (finite, or infinite, or bi-infinite) sequence of elements. The goal is to choose $A_i\in \{A_k: 1\le k\le 3\}$  for each pair  $(g_i, g_{i+1})$ such that
\begin{equation}\label{bddprojIJ3}
\max\{\proj_{A_i}([o,
g_i^{-1}o]), \proj_{A_i}([o, g_{i+1}o])\} \le  \tau.
\end{equation}
and 
\begin{equation}\label{distinctIJEQ}
A_i \ne g_{i+1}A_{i+1}
\end{equation} for verifying the condition  ($\mathbf{LL2}$).

  We start with a fixed pair $(g_i, g_{i+1})$, for which $A_i$ is chosen as above satisfying (\ref{bddprojIJ2}) so (\ref{bddprojIJ3}) is immediate. For subsequent pairs, we need to be careful when applying (\ref{bddprojIJ2}) to obtain $A_i \ne g_{i+1}A_{i+1}$ for the following reason. It is possible that $g_{i+1}$ belong to $E(h_i)$. If it happens, the choice of $A_{i+1}$ by application of (\ref{bddprojIJ2}) may not satisfy $A_i \ne g_{i+1}A_{i+1}$. A simple example to bear in mind is that $G$ splits a direct product of two groups one of which is finite.  The next paragraph thus treats this possibility.
 
If $g_{i+1}$ does not belong to $E(h_i)$, then by (\ref{bddprojIJ2})  there exists $A_{i+1}\in \{A_k: 1\le k\le 3\}$ such that (\ref{bddprojIJ3}) holds for the index ``$i+1$'' and consequently, $A_i \ne g_{i+1} A_{i+1}$. Otherwise, assume that $g_{i+1}$ lies in $E(h_i)$, then by (\ref{bddprojIJ3}), we have $d(o, g_{i+1}o)\le \tau$ so $\proj_{A_{i}}([o, g_{i+1}^{-1}o])\le \tau.$  We choose $A_{i+1}\in \{A_k: 1\le k \le 3\}\setminus A_i$ which exists by the Claim above with the following property 
$$
\proj_{A_{i+1}}([o, g_{i+2}o])\le \tau
$$
which shows  (\ref{bddprojIJ3}) for the case ``$i+1$'' and $A_i \ne g_{i+1}A_{i+1}.$ Hence, the property (\ref{distinctIJEQ}) is fullfilled.

      In this manner, we construct an admissible path $\gamma$  labeled by $(\cdots, g_i, f_i, g_{i+1}, f_{i+1}, \cdots)$, where $f_i \in F$ is the chosen element as above from   $E(h_i)$. Since $A_i \ne g_{i+1}A_{i+1},$ the condition  ($\mathbf{LL2}$) follows from bounded intersection of $\mathbb F$. The condition ($\mathbf{BP}$) follows from (\ref{bddprojIJ3}),  and ($\mathbf{LL1}$) by the choice of $f_i$ with $d(o, f_io)>D$.  Hence, the path $\gamma$ is $(D, \tau)$-admissible with respect to the contracting system $\mathbb F$.

Let $\epsilon_0=\epsilon(\tau), c=c(\tau)>0$ be given by Proposition \ref{admissible}, from which all assertions of this lemma follows as a consequence.

Now, it remains to prove the statement (4). We are looking at their associated admissible paths $\gamma$ and $\gamma'$ with the same endpoints by the hypothesis. If $f_1\ne f_1' \in F$, then by definition of $F$, they belong to different subgroups in $\{E(h_k): 1\le k\le 3\}$. Denote by $X_1, X_1'$ their axis sets so $f_1o\in X_1, f_1'o\in X_1'$. Connect the two endpoints of $\gamma$ (or $\gamma'$) by a geodesic $\alpha$. By the statement (3), we have $d(go, \alpha),  d(gf_1o, \alpha), d(gf_1'o, \alpha)\le \epsilon_0$ for $g:=g_1=g_1'$ by assumption.  We thus obtain a constant $\sigma$ depending on $\epsilon_0$ from the quasi-convexity of $X_1$ and $X_1'$ that a subsegment of $\alpha$ of length at least $\min\{d(o, f_1o), d(o, f_1'o)\}$ is contained in a $\sigma$-neighborhood of both $gX_1$ and $gX_1'$. Hence, it suffices to make $D>\mathcal R(\sigma)$ larger so the bounded intersection of $X_1, X_1'$ would imply a contradiction. The constant $D$ is still uniform, and so the choice of $F$ can be made for the statement (4) as well. Hence all statements are proved. 
\end{proof}

\begin{conv}\label{ConvExtensionLemma}
Choose three pairwise independent contracting elements $f_i$ $(1\le i \le 3)$ in $G$ so that 
$\mathbb F=\{g\ax(f_i):   g\in G \}$ is a contracting system with bounded intersection. Let $F$ be a finite set and  $\epsilon_0, \tau, D>0$ constants supplied by  lemma  \ref{extend3}. 
\end{conv}

\paragraph{\textbf{The extension map.}}  In order to facilitate the use of extension lemmas, it is useful to construct a kind of maps as described as follows.

Given an alphabet set $A$,  denote by $\mathbb W(A)$ the set of all (finite) words over $A$.  Consider an evaluation map $\iota: A \to G$. We are going to define an \textit{extension map} $\Phi: \mathbb W(A) \to G$ as follows: given a word $W=a_1a_2\cdots a_n \in \mathbb W(A)$, set $$\Phi(W)=\iota(a_1) \cdot f_1\cdot\iota(a_2)\cdot f_2\cdot \cdots \cdot \iota(a_{n-1})\cdot f_{n-1} \cdot\iota(a_n) \in G,$$ where $f_i\in F$ is supplied by the extension lemma \ref{extend3} for each pair $(a_i, a_{i+1})$. The product form as above of $\Phi(W)$ labels a $(D, \tau)$-admissible path as follows
$$
\gamma=\path{\iota(a_1)} \cdot\path{f_1}\cdot\path{\iota(a_2)}\cdot \path{f_2}\cdot \cdots \cdot \path{\iota(a_{n-1})}\cdot \path{f_{n-1}} \cdot\path{\iota(a_n)}.
$$

\begin{lem}\label{extensionmap}
For any $\Delta>0$, there exists $R>0$ with the following property.  

Let $W=a_1a_2\cdots a_n$ and $W'=a'_1a'_2\cdots a'_n$ be two words in $\mathbb W(A)$ such  that $\Phi(W)=\Phi(W')$.  If $|d(o, \iota(a_1)o)-d(o, \iota(a'_1)o)|\le \Delta$, then $d(\iota(a_1)o, \iota(a'_1)o)\le R$. 

In particular, if $\iota: A\to G$ is injective such that $\iota(A)o$ is $R$-separated in $A(o, L, \Delta)$ for some $L>0$, then the extension map $\Phi$ defined as above is injective with the image consisting entirely of   contracting elements.
\end{lem}
\begin{proof}
Consider the $(D, \tau)$-admissible path labeled by $\Phi(W')$:
$$
\gamma'=\path{\iota(a'_1)} \cdot\path{f'_1}\cdot\path{\iota(a'_2)}\cdot \path{f'_2}\cdot \cdots \cdot \path{\iota(a'_{m-1})}\cdot \path{f'_{m-1}} \cdot\path{\iota(a'_m)}.
$$
By Proposition \ref{admissible}, a common geodesic $\alpha=[o, \Phi(W)o]$ $\epsilon_0$-fellow travels both $\gamma$ and $\gamma'$ so we have $$d(\iota(a_1)o, \alpha),\; d(\iota(a'_1)o, \alpha) \le \epsilon_0.$$ By the hypothesis, it follows that $|d(o, \iota(a_1)o)-d(o, \iota(a'_1)o)|\le \Delta$. A standard argument shows that $d(go, g'o) \le 2(2\epsilon_0+\Delta)$.  Setting $R:=2(2\epsilon_0+\Delta)$ thus completes the proof of the first assertion. 

By (4) of Lemma \ref{extendinfty}, the injectivity part of the last statement follows as a consequence of the first one. Observe that the set of powers $\{\Phi(W)^n: n\in \mathbb Z\}$ labels a $(D, \tau, L, \Delta)$-admissible path and thus is contracting by Proposition \ref{saturation}. By definition,   the element $\Phi(W)$ is contracting, thereby concluding the proof of the lemma.
\end{proof}

The construction of an extension map is by no means unique. Here is another way to construct it. This will be a key ingredient to prove the existence of large free sub-semigroups in \textsection \ref{Section3}. 
\begin{lem}\label{largesemifree}
There exist  a finite set $F$ in $G$ and for any $\Delta>0$, there exists $R=R(\Delta)>0$ with the following property.  

Let $Z$ be an $R$-separated subset in $A(o, L, \Delta)$ for any given $L>0$. Then there exist an   element $f\in F$ and a subset $A\subset Z$ of cardinality greater than $2^{-4}  \sharp Z$ such that the extension map $\Phi: \mathbb W(A)\to G$ given by $$a_1a_2\cdots a_n \to a_1f^{\epsilon_1}a_2f^{\epsilon_2}\cdots a_nf^{\epsilon_n}$$ for any   $\epsilon_i\in \{1,2\}$  is injective with the image consisting entirely of  contracting elements.
\end{lem}
\begin{proof}
Let $\tau>0$ given by Lemma \ref{SeparationClaim}, and   $D=D(\tau)$ given by Proposition \ref{admissible}.    The idea of the proof is to find an appropriate element $f$ such that  for each $\epsilon_i\in \{1,2\}$, the element $a_1f^{\epsilon_1}a_2f^{\epsilon_2}\cdots a_nf^{\epsilon_n}$ labels a $(D, \tau)$-admissible path. The injectivity statement then follows the same line as in Lemma \ref{extensionmap}.

The first observation is follows:  there exist  a pair $(A_i, A_j)$ and a subset $Z'\subset Z$ such that $4\sharp Z'\ge \sharp Z$, and for each $g\in Z'$ we have 
$$
\max\{\proj_{A_i}([o, go]), \proj_{A_j}([o, go])\}\le \tau.
$$
Indeed, this is acheived by applying Lemma \ref{SeparationClaim} twice: for  a fixed pair, say $(A_1, A_2)$, we  first apply it for every $g\in Z$, then there exists half of elements $Z'$ in $Z$ and  one, say $A_i$, of $(A_1, A_2)$ such that $\proj_{A_i}([o, go]) \le \tau$ for all $g\in Z'$. On the second time, applying   Lemma \ref{SeparationClaim}   to $\{A_1, A_2, A_3\}\setminus A_i$, we reduce $Z'$  in half and find another, say $A_j\ne A_i$ such that $\proj_{A_j}([o, go]) \le \tau$ for all $g\in Z'$. This thus completes the proof of the observation.

Repeating the above argument to the pair $(A_i, A_j)$ and all elements $g\in Z'^{-1}$. We eventually find  $A \subset Z'$ and $A_k\in \{A_i, A_j\}$ such that $4\sharp A\ge \sharp Z'$ and 
\begin{equation}\label{BothBPEQ}
\forall g\in A, \; \max\{\proj_{A_k}([o, go]), \proj_{A_k}([o, g^{-1}o])\}\le \tau.
\end{equation}

The Condition (\textbf{BP}) is verified by (\ref{BothBPEQ}). It suffices to show that $g A_k \ne A_k$ for any $g\in A$. Indeed, if $gA_k=A_k$, then (\ref{BothBPEQ}) implies   $d(o, go)\le \tau$. Since $A$ is $R$-separated, we choose $R>\tau$ so that $g A_k \ne A_k$ for any $g\in A$. Choose an element $f_k$ from the corresponding subgroup of $A_k$ such that $d(o, fo)>D$ and $d(o, f^2o)>D$.   Consequently, we have shown that $a_1f^{\epsilon_1}a_2f^{\epsilon_2}\cdots a_nf^{\epsilon_n}$ labels a $(D, \tau)$-admissible path. By Proposition \ref{admissible}, it is thus a quasi-geodesic. Moreover, noting  that $a_i\in A(o, L, \Delta)$, we see that it is $(D, \tau, L, \Delta)$-admissible path by Definition \ref{uniformadmissible}. Hence, the path is  contracting by Proposition \ref{saturation}.
 This shows that $a_1f^{\epsilon_1}a_2f^{\epsilon_2}\cdots a_nf^{\epsilon_n}$ is a contracting element, thereby completing the proof of the lemma.
\end{proof}

\paragraph{\textbf{Positive density of contracting elements}} By the above proof, we can prove the following lemma, which owns its existence to a recent result of M. Cumplido and B. Wiest in \cite[Theorem 2]{CWiest}. Their result was proved for mapping class groups, but our result works for any sufficiently large subgroup as well as any other proper action with a contracting element.
\begin{lem}\label{contractingelemisclose}
Let $G$ be a group acting  properly on a geodesic metric space $(\mathrm Y, d)$ with  a contracting element. Then there exists a finite set of elements $F$ with the following property.  For any $g\in G$ there exists $f\in F$ such that $gf$ is a contracting element.
\end{lem} 
\begin{proof}
Following the same line as in Lemma \ref{largesemifree}, we prove that for any $g \in G$, there exists $A_k\in \{A_1, A_2, A_3\}$ such that (\ref{BothBPEQ}) holds. If   $g\notin E(h_k)$, then the condition (\textbf{BP})  is satisfied: $gA_k\ne A_k$.  Thus, $af$ is a contracting element for some $f\in E(h_k)$.  Now we consider the case $g\in E(h_k)$. By Lemma \ref{elementarygroup},    the contracting subgroup $\langle h_k\rangle$ is of finite index in $E(h_k)$, so the conclusion is already satisfied for $E(h_k)$. Therefore, the proof is complete.
\end{proof}

With Lemma \ref{contractingelemisclose}, the following result follows by an argument as in \cite[Corollary]{CWiest}. Simutaneously, this is also obtained independently by M. Cumplido \cite{Cumplido} for  Artin-Tits groups on certain hyperbolic spaces. 
 
\begin{prop}[Positive density of contracting elements]\label{Positivedensity}
Under the hypothesis as Lemma \ref{contractingelemisclose}. If $G$ is generated by a finite set, then for $R\gg 0$,
$$
\frac{\sharp \{g \in N(1, R): \text{g is contracting}\}}{\sharp N(1, R)}>0, 
$$ where the ball $N(1, R)$ is defined using the corresponding word metric.
\end{prop} 
\begin{rem}
This strengthens a similar result in a previous version of this paper, stating that the growth rate of contracting element equals to $\e G$ computed using word metric. 

 We remark that a similar statement holds for loxodromic elements in a group $G$ acting on a hyperbolic space with WPD loxodromic elements. 
 Indeed, a loxodromic WPD element in a (not necessarily proper) group action gives rise to a contracting subgroup with bounded intersection (see Theorem 6.8 in \cite{DGO}). This is the only ingredient of the above two lemmas so their proofs show the positive density of loxodromic elements for any word metric. 
\end{rem}

\paragraph{\textbf{Finiteness of dead-end   depth}}
To conclude this subsection, we mention a straghtforward application of the  extension lemma to dead-ends.  In Cayley graphs, the dead-depth of a  vertex roughly describes  the ``detour'' cost to get around a dead-end (i.e.  an endpoint of a maximal unexendable geodesic).  In this context, the finiteness of dead-end depth was known in hyperbolic groups \cite{Bog}, and in relative case \cite{YANG6}.  

Let $(\mathrm Y, d)$ be a proper geodesic metric space with a basepoint $o$. The \textit{dead-end depth} of a point $x$ in $\mathrm Y$ is the infimum of a non-negative real number $r\ge 0$ such that there exists a geodesic ray $\gamma$ satisfying $\gamma_-=o$ and $\gamma \cap B(x, r) \ne \emptyset$.   If the dead-end depth is non-zero, then $g$ is called a \textit{dead end} (i.e. the geodesic $[o, x]$ could not be further extended). Such elements exist, for instance, in the Cayley
graph of $G \times \mathbb Z_2$ with respect to a particular generating set. See \cite{GriH} for a brief discussion and
references therein.

The first, straightforward  consequence of the extension lemma is the following. 

\begin{prop}\label{deadend}
Let $o$ be a fixed basepoint. If $G$ acts properly on $\mathrm Y$ with a contracting element, then the dead-end depth of all points in $Go$ is uniformly finite.
\end{prop}
As a corollary,   the dead-end depth of an element in $\mathrm{Gr}'(1/6)$-labeled graphical small cancellation group as described   in \ref{Gr1/6} is uniformly finite.   This result appears to be not recorded in the literature.

It is worth to point out that the conclusion of the extension lemma  holds, as long as there exist at least two independent contracting elements. In particular, the lemma applies to the action of $\mathbb {Mod}$ on a \textit{curve graph} (\cite{MinM}\cite{MinM2}) associated to orientable surfaces with negative Euler number. Consequently, the dead-depth of vertices in curve graph is uniformly finite. This result was proved by  Schleimer \cite{Slmer}, and by  Birman and   Menasco \cite{BirmanMen} with a   more precise description. Here our arguments are completely general, without appeal to the   construction of curve graphs.

\subsection{A critical gap criterion}
A proper action of $G$ on a metric space $(\mathrm Y, d)$ naturally induces a proper left-invariant pseudo-metric $d_G$ on $G$ as follows: 
$$
\forall g, h\in G: \;d_G(g, h)=d(go, ho)$$
for a basepoint $o\in \mathrm Y$. Using the pseudo-metric $d_G$, we can define the ball set $N(o, n)$, the critical exponent $\e \Gamma$ of a subset $\Gamma$ equivalently. To emphasize the metric, we use the notation $\omega_{d_G}(\Gamma)$ only in this subsection. In other places, the metric should be clear in context.

To present the criterion,  one needs to introduce the  Poincar\'e series 
$$\PS{\Gamma, d_G} = \sum\limits_{g \in \Gamma} \exp(-s\cdot d_G(1, g)), \; s \ge 0,$$
which clearly  diverges for $s<\omega_{d_G}(\Gamma)$ and converges for $s>  \omega_{d_G}(\Gamma)$.

Dalbo et al. \cite{DPPS} presented a very useful criterion to differentiate the critical exponents of two Poincar\'e series, which is the key tool to establish growth-tightness. We formulate it with   purpose to exploit the extension map.


Let $A, B$ be two sets in $G$. Denote by $\mathbb W(A, B)$ the set of all words over the alphabet set $A\sqcup B$ with letters alternating in $A$ and $B$. We consider a left-invariant pseudo-metric $d_G$ on $G$ (which might not be coming from the pullback via a proper action). 
\begin{lem}\label{degrowth}
Assume that there exists an injective  map $\iota:  \mathbb W(A)\to  \mathbb W(A, B)$ such that the evaluation map $\Phi: \mathbb W(A, B) \to G$ is injective on the subset $\iota(\mathbb W(A))$ as well. Denote $X:=\Phi(\iota(\mathbb W(A)))\subset G$. If $B$ is finite, then $\PS{A, d_G}$ converges at $s=\omega_{d_G}(X)$. In particular,   we have the critical gap:
  $$\omega_{d_G}(X)
> \omega_{d_G}(A)$$ provided that $\PS{A, d_G}$ diverges at $s=\omega_{d_G}(A)$ 
\end{lem}

\begin{proof}
Since $\Phi: \mathbb W(A, B) \to G$ is injective, each element in the image $X$ has a unique  alternating product form over $A\sqcup B$. For a word $W=a_1b_1a_2\cdots a_nb_n \in \mathbb W(A, B)$, we have $$
d_G(1, a_1b_1\cdots a_nb_n) \le \sum_{1\le i\le n} \big (d_G(1, a_i)+D\big),
$$ 
where $D:=\max \{d_G(1, b): {b\in B}\}<\infty$ since $B$ is a finite set.
As a consequence, we estimate the Poincar\'e series  $\PS{X, d_G}$  of $X$ as follows:
$$
\begin{array}{rl}
&\sum\limits_{g \in X} \exp(-s\cdot d_G(1, g)) \\
\ge&  \sum\limits_{n =1}^{\infty} \left(\sum\limits_{a_1, \ldots, a_n \in A} \exp(-s\cdot d_G(1, a_1 b_1 \cdots a_n  b_n)) \right)\\
\ge &   \sum\limits_{n =1}^{\infty} \left(\sum\limits_{a \in A} \exp(-s\cdot d_G (1, a))\right)^n\cdot   \exp(-sD)^n.
\end{array}
$$

By way of contradiction, assume that $\PS{A, d_G}$ diverges at $s=\omega_{d_G}(X)$ so   there exists some $s_0 >
\omega_{d_G}(X)$ such that $\sum\limits_{a \in A} \exp(-s_0\cdot d_G (1,
a))> 1/D$. Consequently, the above  series $\PS{X, d_G}$ diverges at $s=s_0$, so  implies that  $\omega_{d_G}(X) > s_0$. This  contradiction concludes the proof of lemma. 
\end{proof}
\begin{rem}
From now on, we shall always consider the metric $d_G$ as the pullback of the metric $d$ on $\mathrm Y$ via the proper action. Hence, the subindex $d_G$ is omited for simplicity. 
\end{rem}

Finally, the following lemma will be frequently used, cf. \cite[Lemma 3.6]{PYANG}. Recall that a metric space $(X, d)$ is \textit{$C$-separated} if $d(x, x')>C$ for any  two $x\ne x'\in X$.
\begin{lem}\label{sepnet}
Let $(\mathrm Y, d)$ be a proper metric space on which a group $G \subset \isom(\mathrm Y)$ acts properly. For any orbit $Go$ ($o\in \mathrm Y$) and $R>0$ there exists a constant $\theta=\theta(Go, R)
>0$ with the following property.

For  any finite set $X$ in $Go$, there exists an $R$-separated 
subset $Z \subset X$ such that $\sharp Z > \theta \cdot \sharp X$.
\end{lem}

\section{Large free sub-semigroups}\label{Section3}
This section is devoted to the proof of   \ref{LargeFreeThm}.

Let $(\mathbb F, F, \epsilon_0, \tau, D)$ be given by Convention \ref{ConvExtensionLemma}.  

Recall  $$\limsup_{n\to \infty}\frac{\log \sharp A(o, n, \Delta)}{n} = \e G.$$

We first prove the existence of large free semigroups. 

\begin{lem}\label{largeAtree}
For any $0<\omega <\e G$, there exists a free sub-semigroup $\Gamma$ such that  
$\e {\Gamma}>\omega$.  
\end{lem}
\begin{proof}
Fix $\Delta>0$. Let $R=R(\Delta)>0$ given by Lemma \ref{largesemifree} and $\theta=\theta(R)$   obtained by Lemma \ref{sepnet}.  For given $\omega <\e G$, choose a large number $k_\omega>4$ and then a constant $\omega\le \delta<\e G$ such that 
\begin{equation}\label{defdelta}
\delta\ge \frac{(k_\omega+\Delta)\omega- 2^{-k_\omega}\theta}{k_\omega} \ge \omega,
\end{equation}
and at the same time, the following is true:
$$\sharp A(o, k_\omega, \Delta)>\exp(  k_\omega\delta).$$
Indeed, when $k_\omega$ is sufficiently large, the fraction in (\ref{defdelta}) lies between $\omega$ and $\e G$. 

By Lemma \ref{sepnet}, there exists $Z\subset A(o, k_\omega, \Delta)$ such that $Zo$ is $R$-separated set and  
\begin{equation}\label{defZset}
\sharp Z\ge \theta\cdot  \exp(  k_\omega\delta).
\end{equation}
Furthermore, the lemma \ref{largesemifree} gives a subset $A$ of $Z$ with $2^4\sharp A\ge \theta \sharp Z$ and an element $f\in G$ such that $$\Phi: \mathbb W(A)\to G$$ defined by $$a_1a_2\cdots a_n \to a_1fa_2f\cdots a_nf$$ is injective. Setting $S=A\cdot f$, the injectivity of $\Phi$ is amount to saying that the semigroup $\Gamma$ generated by $S$ is free with base $S$.

To complete the proof, it suffices to estimate the critical exponent. Note that the ball $N(o, n \cdot (k_\omega+\Delta))$ contains at least a set of elements of $\Phi(\mathbb W_n)$ where $\mathbb W_n$ is the set of words of length $n$ in $\mathbb W$. By the injectivity of $\Phi$, we have $\sharp \Phi(\mathbb W_n)\ge  (\sharp A)^n$. So by (\ref{defZset}) we obtain: 
\begin{equation}\label{cardTn}
\sharp \big(N(o, n \cdot (k_\omega+\Delta))\cap \Gamma\big) \ge  \sharp \mathbb W_n \ge  (\sharp A)^n \ge  ( {2^{-4}\theta})^n\exp(n \cdot k_\omega \delta ).
\end{equation}
By definition of $\delta$ (\ref{defdelta}), the critical exponent can be estimated below:
$$
\e {\mathcal T} \ge \limsup_{n\to \infty}\frac{\log \sharp  \big(N(o,n \cdot (k_\omega+\Delta))\cap \mathcal T\big) }{n\cdot (k_\omega+\Delta)} \ge  \frac{2^{-4}\theta+ k_\omega\delta}{k_\omega+\Delta}  \ge \omega.
$$
The proof is complete. 
\end{proof}

\paragraph{\textbf{Quasi-geodesic contracting tree}} The Cayley graph of a semigroup $G$ with respect to a generating set $S$ can be defined in the same way as the case of groups.     The vertex set is $G$, and two vertices $g_1, g_2 \in G$ are connected by an oriented edge labeled by $s\in S$ if $g_2=g_1s$. Consider the orbital map
$$
\Psi: \Gx \to \mathrm Y
$$
which sends vertices $g$ to $go$ and the edges $[1,s]$ to a geodesic $[o, so]$ and other edges by translations. If the semigroup $G$ is freely generated by $S$, then the Cayley graph is a tree.  Moreover, the image $\Psi(\Gx)$ is a quasi-geodesic rooted tree so that each branch is contracting. This is just a re-interpretation of Lemma \ref{largesemifree}. Indeed, each branch was proved to be a $(D, \tau, L, \Delta)$-admissible path, so it is contracting by Proposition \ref{saturation}. And the quasi-geodesic statement follows from Proposition \ref{admissible} for a $(D, \tau, L, \Delta)$-admissible path.

To complete the proof of  \ref{LargeFreeThm}, it suffices to verify the following.
\begin{lem}
The free semigroup $A:=\Gamma_n$  is growth-tight.
\end{lem}
\begin{proof}
To prove growth-tightness of $A$, we  shall apply the criterion \ref{degrowth} to $A$ and $B:=\{f\}$.  Consider the set $\mathbb W(A, B)$ of words with letters alternating in $A$ and $B$.  Hence, each word in $\mathbb W(A, B)$ has the form $a_1f^{\epsilon_1}a_2f^{\epsilon_2}\cdots a_nf^{\epsilon_n}$ for $\epsilon_i\in \{1, 2\}.$
As in the proof of Lemma \ref{largeAtree},  we obtain by Lemma \ref{largesemifree}    that $\mathbb W(A, B)$ is sent into $G$ injectively. 

Hence, it suffices to verify that the Poincar\'e series $\PS{A}$ of $A$ diverges at $s=\e {A}$. In turn, we shall prove that $\sharp S_{n, \Delta} \ge \exp(n \cdot \e{A})$, where $S_{n, \Delta}:=A\cap A(o, n, \Delta)$. Apparently, this implies the divergence of $\PS{A}$ at $s=\e A$.  

Since each branch of the tree associated to $\Gamma_n=A$ is contracting with a uniform contraction constant, it is thus a   quasi-geodesic path. A similiar argument as in the Case 1 of proof of Lemma \ref{SubMulSet}   shows  that 
$\sharp S_{n+m, \Delta} \le \sharp S_{n, \Delta} \cdot \sharp S_{m, \Delta}$ for some $\Delta>0$.  By Feketa's Lemma, we have $\sharp S_{n, \Delta} \ge \exp(n \cdot \e{A})$, thereby completing the proof of the result.
\end{proof}

So  all the statements in  \ref{LargeFreeThm} are proved.

\section{Growth-tightness theorem}\label{Section4}


Recall that a  subset $X$ in $G$ is \textit{growth-tight} if $\e X<\e G$. The union of two growth-tight sets is  growth-tight.  
The main result,   \ref{GrowthTightThm}, of this section shall provide a class of  growth-tight sets. These growth-tight sets are closely related to a notion of a barrier we are going to introduce now.  





\begin{defn}\label{barriers}
Fix constants $\epsilon, M>0$ and a set $P$ in $G$. \begin{enumerate}
\item(Barrier/Barrier-free geodesic)
Given $\epsilon>0$ and $f\in P$, we say that a geodesic $\gamma$ contains an \textit{$(\epsilon, f)$-barrier}   if there exists    a element $h \in G$ so that 
\begin{equation}\label{barrierEQ}
\max\{d(h\cdot o, \gamma), \; d(h\cdot fo, \gamma)\}\le \epsilon.
\end{equation}
If 
 no such $h \in G$ exists so that (\ref{barrierEQ}) holds, then $\gamma$ is called \textit{$(\epsilon, f)$-barrier-free}.  

 Generally, we say $\gamma$ is \textit{$(\epsilon, P)$-barrier-free} if it is $(\epsilon,f)$-barrier-free for some $f\in P$. An obvious fact is that  any subsegment of $\gamma$ is also $(\epsilon, P)$-barrierr-free.
\item(Barrier-free element)
An element $g\in G$ is \textit{$(\epsilon, M, P)$-barrier-free} if there exists  an $(\epsilon, P)$-barrier-free geodesic between $B(o, M)$ and $B(go, M)$. 
\end{enumerate}
\end{defn}

The definition \ref{StatConvex} of statistically convex-cocompact actions in Introduction replies on the formulation of a concave region. Let us repeat it here for convenience. For  constants $M_1, M_2\ge 0$, let $\mathcal O_{M_1, M_2}$ be the set  of elements $g\in G$ such that there exists some geodesic $\gamma$ between $B(o, M_2)$ and $B(go, M_2)$ with the property that  the interior of $\gamma$ lies outside $N_{M_1}(Go)$.

In applications, since $\mathcal O_{M_2, M_2} \subset \mathcal O_{M_1, M_2}$, we can assume that $M_1=M_2$ and henceforth, denote $\mathcal O_M:=\mathcal O_{M, M}$ for easy notations.

\subsection{Sub-multiplicative inequality}
In this subsection, we establish a variant of sub-multiplicative inequality for SCC actions. This idea was first introduced by F. Dal'bo et al. \cite{DPPS}. We here adapt their argument in our context.

For a constant $\Delta>0$  and a subset $P\subset G$, consider the annulus sets as follows $$\mathcal V_{\epsilon, M, P}(n, \Delta):=\mathcal V_{\epsilon, M, P}\cap A(o, n, \Delta),$$ 
and 
$$\mathcal O_{M}(n, \Delta):=\mathcal O_{M}\cap A(o, n, \Delta) \cup \{1\}$$
where $1\in \mathcal O_{M}(n, \Delta)$ is used in the proof of the following lemma.

\begin{lem}\label{SubMulSet}
There exists $\Delta>0$ such that the following holds  
\begin{equation}\label{threeunion}
\sharp \mathcal V_{\epsilon, M, P}(n+m, \Delta) \le \sum_{1\le k \le n; 1\le j \le m} \sharp \mathcal V_{\epsilon, M, P}(k, \Delta) \cdot \sharp \mathcal O_{M}(n+m-k-j, 2\Delta)  \cdot   \sharp \mathcal V_{\epsilon, M, P}(j, \Delta) 
\end{equation}
for any $n, m\ge 0$. 
\end{lem}

\begin{proof}
By definition, an element $g$ belongs to $\mathcal V_{\epsilon, M, P}$, if and only if, there exists a geodesic $\gamma=[x, y]$ for some  $x\in B(o, M)$ and $y\in B(go, M)$ such that $\gamma$ is $(\epsilon, M, P)$-barrier-free. Note that any subpath of $\gamma$ is $(\epsilon, P)$-barrier-free as well.

Set $\Delta=4M$. Let $g\in\mathcal V_{\epsilon, M, P}(n+m, \Delta)$, so there exists a geodesic $\gamma=[x, y]$ with properties stated as above. Since $|d(o, go)-n-m|\le \Delta$ and thus $|d(x, y)-n-m|\le \Delta+2M=6M$, we can write $$d(x, y) = m+n + 2\Delta_1,$$ for some $|\Delta_1| \le 3M$. Consider a point $z$ in $[x, y]$ such that $d(o, z)=n+\Delta_1$. 

\textbf{Case 1.} Assume that $z\in N_M(Go)$ so there exists $h\in G$ such that $d(z, ho)\le M$. Thus,
$$\max\{|d(o, ho)-n|,\; |d(ho, go)-m|\}  \le 4M\le\Delta.$$
Note that $[x, z]_\gamma$ and $[z, w]_\gamma$ as subsegments of $\gamma$ are $(\epsilon, P)$-barrier-free, so we obtain that $h\in \mathcal V_{\epsilon, M, P}(n, \Delta)$ and $h^{-1}g\in  \mathcal V_{\epsilon, M, P}(m, \Delta)$.  Since we included $1\in \mathcal O_M(0, 2\Delta)$ by definition,   the element $g$  belongs to the union set in the right-hand of (\ref{threeunion}).

\textbf{Case 2.} Otherwise, consider the maximal open segment $(z_1, z_2)$ of $[x, y]$ which contains $z$ but lies outside $N_M(Go)$. Hence, there exist $g_1, g_2\in G$ such that $d(z_i, g_io)\le M$ for $i=1, 2$.  By definition,  we have $t:=g_1^{-1}g_2 \in \mathcal O_M$. Similarly as above, we see that $$g_1 \in \mathcal V_{\epsilon, P}(k, \Delta), \; g_2^{-1}g \in\mathcal V_{\epsilon, M, P}(l, \Delta)$$ where $k:=d(o, z_1)$ and $l:=d(z_2, go)$. On the other hand, 
$$
\begin{array}{ll}
d(z_1,z_2)&= d(x, y) -d(x, z_1)-d(z_2, y) \\
&\simeq_{2M} d(x, y) -d(o, z_1)-d(z_2, go)\\
&\simeq_{2M+\Delta_1} n+m-k-l.
\end{array}
$$
It follows from   $|d(g_1o,g_2o)- d(z_1,z_2)|\le 2M$ that 
$$
|d(o, to)-n-m+k+l|\le 2\Delta.
$$

That is, $t\in O_M(n+m-k-j, 2\Delta)$. Writing $g=g_1 \cdot t \cdot (g_2^{-1}g)$ completes the proof of (\ref{threeunion}) in the Case (2). The lemma is proved.
\end{proof}

For $\omega>0$, we define $$a^\omega(n, \Delta)= \exp(-\omega n) \cdot \sharp \mathcal V_{\epsilon, M, P}(n, \Delta).$$

\begin{lem}\label{SubMulInQ}
Assume that   $\e {\mathcal V_{\epsilon, M, P}} > \e{\mathcal O_M}$.   Then given any $\e {\mathcal V_{\epsilon, M, P}} > \omega> \e{\mathcal O_M}$, there exist $\Delta, c_0>0$ such that the following holds 
\begin{equation}\label{SubMulEQ}
a^\omega(n+m,\Delta) \le c_0 \left( \sum_{1\le k\le n} a^\omega(k, \Delta) \right) \cdot \left(\sum_{1\le j \le m} a^\omega(j,\Delta) \right),
\end{equation}
for any $n, m \ge 0$. Moreover,   the Poincar\'e series $\PS{\mathcal V_{\epsilon, P}}$  diverges at $s= \e{\mathcal V_{\epsilon, P}}$.
\end{lem}
\begin{proof}
Note that there exists $c_0>0$ such that $\sharp \mathcal O_M(i, 2\Delta)\le c_0 \exp(i \cdot \omega )$ for any $i\ge 1$.  Consequently, a re-arrangement of (\ref{threeunion}) gives rise to the form of   (\ref{SubMulEQ})  (see \cite[Proposition 4.1]{DOP} for instance). The ``moreover" statement follows by \cite[Lemma 4.3]{DOP}.
\end{proof}
\begin{rem}
The above proof of Lemma \ref{SubMulSet}    still works for the set $\mathcal V_{\epsilon, P}(n, \Delta)$ replaced by $A(o, n, \Delta)$: in fact, the proof gets greatly simplified.   So the inequality (\ref{SubMulEQ}) holds also for $a^\omega(n, \Delta):= \exp(-\omega n) \cdot \sharp A(o, n, \Delta)$. See  the proof of Theorem \ref{purexpgrowth} below. 
\end{rem}

\subsection{The main construction}
For given $g\in G$ and  $\epsilon>0$, denote by $\mathcal V_{\epsilon, M, g}$ the set of $(\epsilon, M, g)$-barrier-free elements in $Go$ from $o$.  Denote by  $\mathbb W(\mathcal A)$  the free monoid over the alphabet set $\mathcal A:=\mathcal V_{\epsilon,M,  g}$.  

To be clear, we fix some constants at the beginning (the reader is encourage to read the proof first and return here until the constant appears). 
\\
\paragraph{\textbf{Setup}}
Let $\mathbb F$ be a contracting system satisfying Convention \ref{ConvExtensionLemma}. 
\begin{enumerate}
\item
We denote by $C>0$ the contraction constant for the contracting system  $\mathbb F$. For easy notations, we assume that $C$  satisfies  Proposition \ref{Contractions} as well.
\item
Constants $\tau, D  >0$ in admissible paths are given by Lemma \ref{extend3}.  
\item
Let $B>0$ be the constant by Proposition \ref{admissible} for $(D, \tau)$-admissible paths relative to the contracting system $\mathbb F$. 
\item
Choose by Lemma \ref{extend3} a finite set $F$ such that 
\begin{equation}\label{BigF}
d(o, ho)>2B+C+M
\end{equation}  for each $h\in F$. 
\\
\end{enumerate}
 
\paragraph{\textbf{The extension map} $\Phi: \mathbb W(\mathcal A)\to G$}
To  each word $W=w_1 w_2 \cdots w_n \in \mathbb W(\mathcal A)$, we associate an element $\Phi(W)\in G$ defined as follows: 
$$\Phi(w_1  w_2 \cdots  w_n)=w_1 \cdot (f_1 g h_1)\cdot w_2\cdot (f_2 g h_2)\cdot \cdots \cdot (f_{n} g h_n) \cdot w_n$$
where $f_i, h_i\in F$ are chosen by the extension lemma \ref{extend3}.

As before, consider  the  path labeled by $\Phi(W)$ which is  a $(D, \tau)$-admissible path:
$$\gamma = \path{w_1} \cdot (\path {f_1}\cdot \path{g_1} \cdot\path{h_1})\cdot \path {w_2}\cdot (\path{f_2}\cdot \path{g_2}\cdot \path{h_2})\cdot \cdots \cdot (\path {f_{n}}\cdot \path{g_n} \cdot \path{h_n}) \cdot \path{w_n},$$
  Note that the geodesics $\path{g_i}$ are all labeled by the same element $g$.

\begin{lem} \label{injective}
There exist   constants $\epsilon=\epsilon(\mathbb F, M), R = R(\mathbb F, M) > 0$ with the following property.

Let $W=w_1 w_2 \cdots w_n, W'=w'_1 w'_2 \cdots w'_m$  in $\mathbb W(\mathcal A)$ such that
$\Phi(W) = \Phi(W')$. Then  $d(w_1  o, w_1'  o) \le R$.
\end{lem}

\begin{proof}
let us look at their labeled paths:
$$\gamma = \path{w_1} \cdot (\path {f_1}\cdot \path{g_1} \cdot\path{h_1})\cdot \path {w_2}\cdot (\path{f_2}\cdot \path{g_1}\cdot \path{h_2})\cdot \cdots \cdot (\path {f_{n}}\cdot \path{g_n} \cdot \path{h_n}) \cdot \path{w_n},$$
and
$$\gamma' =\path{w'_1} \cdot (\path {f'_1}\cdot \path{g_1'} \cdot\path{h'_1})\cdot \path {w'_2}\cdot (\path{f'_2}\cdot \path{g_2'}\cdot \path{h'_2})\cdot \cdots \cdot (\path {f'_{m}}\cdot \path{g_n'} \cdot \path{h'_m}) \cdot \path{w'_m}.$$ Fix a geodesic $\alpha$ between $o$ and $\Phi(W) o$, where
$\Phi(W)=\Phi(W')$. It is possible that $m\neq n$.  

First, set $$R_0:=2\max\{d(o, fo): f\in F\}+d(o, go)$$ which implies   
\begin{equation}\label{EQ1}
d(o, w_1f_1gh_1o) \le d(o, w_1o)+R_0.
\end{equation}
Without loss of generality, assume that $d(o, w_1'o)\ge d(o, w_1o)$. 

Set $R:=R_0+3(2C+B)+M$. Our goal is to prove that 
\begin{equation}\label{Hypothesis}
d(o, w_1'o)\le d(o, w_1o)+R,
\end{equation} 
which clearly concludes the proof of the lemma. 

By way of contradiction, let us assume that (\ref{Hypothesis}) fails, and go to prove that $w_1'$ contains an $(\epsilon, g)$-barrier. To this end, consider a geodesic $\beta=[x, y]$ where $x\in B(o, M)$ and $y\in B(w_1'o, M)$. 

Since $\gamma$ is $(D, \tau)$-admissible, by Proposition \ref{admissible}, there exists $B>0$ such that 
\begin{equation}\label{BETA12}
\max\{\proj_X(\beta_1), \proj_X(\beta_2)\} \le B
\end{equation} 
where  $\beta_1:=[o, (\path{h_1})_-]_\gamma$ and $\beta_2:=[(\path{h_1})_+, \gamma_+]_\gamma$. See Figure \ref{fig:lem47}.

\begin{figure}[htb] 
\centering \scalebox{0.8}{
\includegraphics{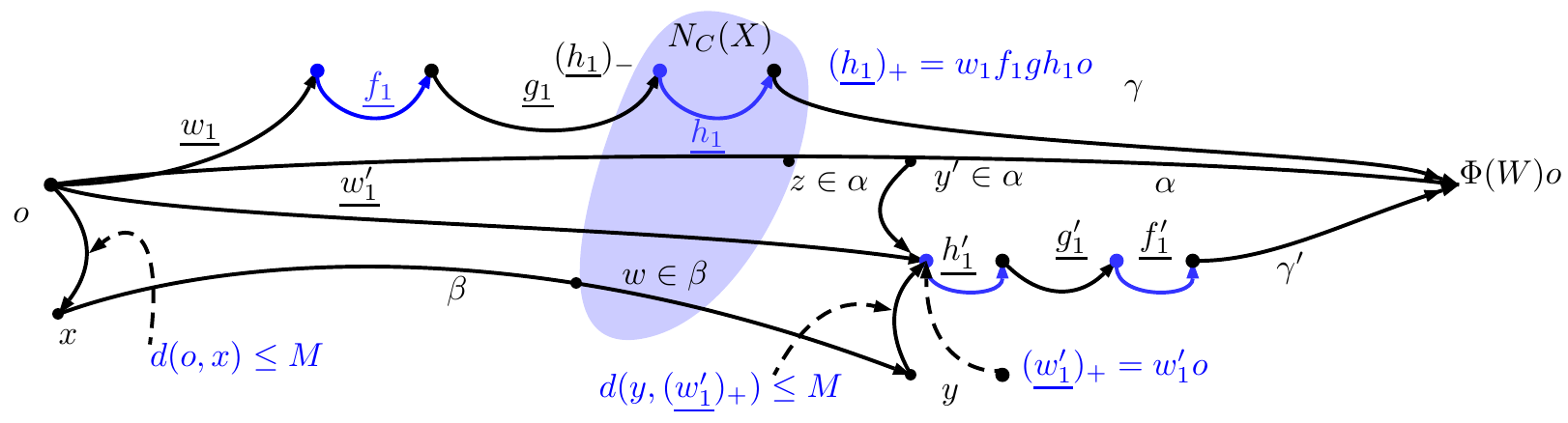} 
} \caption{Proof of Lemma \ref{injective}} \label{fig:lem47}
\end{figure}

Denote by $X$ the contracting set associated to $\path{h_1}$.  Let $z$ be the corresponding exit point of $\alpha$ in $N_{C}(X)$, which exists by Proposition \ref{admissible}. By contracting property and (\ref{BETA12}) we obtain 
\begin{equation}\label{EQ2}
d(z, (\path{h_1})_+)\le \proj_X(\beta_2)+\proj_X([\alpha_+, z]_\alpha)+d(z, X)\le 2C+B.
\end{equation}
Note that $(\path{h_1})_+=w_1f_1gh_1o$. So by (\ref{EQ1}) and (\ref{EQ2}), we have 
\begin{equation}\label{EQ3}
d(o, z)\le    d(o, w_1o)+R_0+2C+B.
\end{equation}

We claim that 
\begin{claim}
$\beta \cap N_C(X)\ne \emptyset$. 
\end{claim}
\begin{proof}[Proof of the Claim]
We consider the contracting set $Y$ associated to $\path{h_1'}$ and the entry point $y'$ of $\alpha$ in $N_C(Y)$. A similar estimate as (\ref{EQ2}) shows $$d((\path{w'_1})_+, y')\le 2C+B.$$
Noticing that $(\path{w'_1})_+=w_1'o$ and  (\ref{Hypothesis}) was assumed to be false, we get the following   from (\ref{EQ3}) and (\ref{Hypothesis}): 
$$
\begin{array}{ll}
d(o, y')-d(o, z)&\ge d(o, w_1'o) -d(o, w_1o) - R_0-2(2C+B)\\
&> 2C+B+M>0,
\end{array}
$$
implying $y'\in [z, \alpha_+]_\alpha$. Since $d(y', y) \le d(y, (\path{w'_1})_+)+d((\path{w'_1})_+, y)\le 2C+B+M$, we have $d(y', z)>d(y', y)$. This shows that the geodesic $[y, y']$ is disjoint with $N_{C}(X)$: indeed, since $z$ is the exit point of $\alpha$ in $N_{C}(X)$, we would obtain $d(y', z)\le d(y', y)$, a contradiction. The contracting property thus implies 
\begin{equation}\label{yyprimeEQ}
 \proj_X([y, y']) \le C.
\end{equation} 

We are now ready to prove the claim: $\beta \cap   N_C(X)\ne \emptyset$. Indeed, if it is false, then it follows $\proj_X(\beta)\le C$ by contracting property. Moreover, since $d(o, x)\le M$, we have $$\proj_X([o, x]) \le M+C,$$  by Proposition \ref{Contractions}.\ref{1Lipschitz}. We now estimate by projection: 
$$
\begin{array}{ll}
\len(\path{h_1})& \le \proj_X(\beta_1)+\proj_X([o, x])+\proj_X(\beta)+\proj_X([y, y'])+\proj_X([y', \alpha_+]_\alpha)+\proj_X(\beta_2)\\
&\le 2B+3C+M,
\end{array}
$$
where (\ref{BETA12}) and (\ref{yyprimeEQ}) are used.
This inequality contradicts to the choice of $h_1\in F$ satisfying (\ref{BigF}). So the claim is proved. 
\end{proof}
 
Let us return to the proof of the lemma. Consider the entry point $w$ of $\beta$ in $N_C(X)$, which exists by the above claim.   So 
$$
\begin{array}{ll}
d((\path{g_1})_+,  w)&\le \proj_X(\beta_1)+\proj_X([o, x]) + \proj_X([x, w]_\beta)+d(w, X) \\
&\le B+M+3C.
\end{array}
$$
 
Similarly, we proceed the above analysis for the contracting set associated to $\path{f_1}$ and  we can prove that $d((\path{g_1})_+, \beta)\le B+M+3C$. 

By setting $$\epsilon=B+M+3C,$$ we have proved that the two endpoints of $\path{g_1}$ lie within at most an $\epsilon$-distance to $\beta$. By definition of barriers, we have that $\beta$ contains an  $(\epsilon, g)$-barrier.   This contradicts to the choice of $w'_1 \in \mathcal A$, where $\mathcal A=\mathcal V_{\epsilon, M, g}$.

In conclusion, we have showed that (\ref{Hypothesis}) is true: $d(w_1o, w_1'o)\le R$, completing the proof of lemma. 
\end{proof}

\subsection{Proof of \ref{GrowthTightThm}}

By definition of a SCC action, there exist a constant $M>0$ such that $\e {\mathcal O_{M}} <\e G$.

For the set $\mathcal A:=\mathcal V_{\epsilon, M, g}$, there exist  $\epsilon=\epsilon(\mathbb F, M)$ and $R=R(\mathbb F, M)$ satisfying the conclusion of Lemma \ref{injective}. 

Without loss of generality, assume  that $\e {\mathcal A}>\e {\mathcal O_M}$ and we shall prove that $\e {\mathcal A} <\e G$.

Let $\mathcal B$ be a \textit{maximal $R$-separated} subset in $\mathcal A$ so that 
\begin{itemize}
\item
for any distinct $a, a'
\in \mathcal B$,  
$d(a o,   a' o)
> R$, and
\item
for any $x \in \mathcal V_{\epsilon, M, g}$, there
exists $a \in \mathcal B$ such that $d(xo,  ao) \le R$. 
\end{itemize}
Taking Lemma \ref{extendinfty}.(4) into account, the map $\Phi: \mathbb W(\mathcal B) \to
G$ defined before Lemma
\ref{injective} is   injective.

On the other hand, an elementary argument shows that $\PS{\mathcal B} \asymp \PS{\mathcal A}$, whenever they are finite, and so $$\e {\mathcal B} = \e {\mathcal A}.$$ By
Lemma \ref{SubMulInQ},  $\PS{\mathcal A}$ and thus $\PS{\mathcal B}$
are divergent at $s = \e {\mathcal A}$.

Consider, $X:=\Phi(\mathbb W(\mathcal B))$, the image of $\mathbb W(\mathcal B)$ under the map $\Phi$ in $G$. The criterion \ref{degrowth} implies $\e {X} > \e {\mathcal B}$ and so $\e{G} \ge
\e {X} > \e {\mathcal A}$. Thus,  the barrier set $\mathcal V_{\epsilon, M, g}$  is growth-tight, thereby concluding the proof of theorem.
\\

For a proper action, the above proof actually shows the following general fact. We consider a weaker notion of \textit{growth-negligible} subsets $X$ in $G$ are proven to be useful in a further study:$$\frac{\sharp (N(o, n)\cap X)}{\exp( \e G n)}\to 0$$
as $n\to \infty$.  

\begin{cor}
Suppose that a group $G$ acts properly on a geodesic space $(\mathrm Y, d)$. Then, under the same quantifiers as \ref{GrowthTightThm}, $\mathcal V_{\epsilon, M, g}$ is a growth-negligible set.
\end{cor} 
\begin{proof}
Indeed, the assumption of SCC actions is used by Lemma \ref{SubMulInQ} to guarantee the divergence of $\PS{\mathcal A}$. Except this place, the proper action suffices to prove Lemma \ref{injective}. So the criterion \ref{degrowth} shows that  $\mathcal V_{\epsilon, M, g}$ is a growth-negligible set.
\end{proof}


\subsection{Some applications}
In this subsection, we collect some sample applications of \ref{GrowthTightThm} by demonstrating some ways to embed interesting subsets into a barrier-free set. More applications shall be presented in the paper \cite{YANG11}.

Firstly, \ref{GrowthTightThm} generalizes the growth-tightness for groups  introduced by Grigorchuk and de la Harpe in \cite{GriH}. 

\begin{cor}\label{growthtightcor}
Under the same assumption as \ref{GrowthTightThm}, we have 
$$
\e {\bar G} < \e G,
$$
for any quotient $\bar G$ of $G$ by an infinite normal subgroup $N$. Here $\e {\bar G}$ is computed with respect to the proper action of $\bar G$ on $\mathrm Y/N$.
\end{cor}
\begin{proof}
Indeed, choose a \textit{shortest} representative $h$ in $G$ for each  element $\bar h=Nh$ in a quotient group: $d(o, ho)=d(o, Nh\cdot o)$. It is clear that the set $\Gamma$ of these representatives has growth rate (computed with metric $d$) equal to $\e {\bar G}$.  It suffices to see that $\Gamma$  is contained in a set of $(\epsilon, M, g)$-barrier-free elements for a fixed ``long'' element $g$ in the kernel $N$. We can first find  an   element $g$ in $N$ such that $d(o, go)>4\epsilon+1$, since $N$ is infinite. 

We claim that $\Gamma \subset \mathcal V_{\epsilon, M, g}$. If not, then the geodesic $\gamma=[o, ho]$ contains an $(\epsilon, g)$-barrier $t\in G$: $d(t\cdot o, \gamma), d(t\cdot ho, \gamma)\le \epsilon$.  Since $N$ is normal, we have $tg=\hat g t$ for some $\hat g\in N$. So 
$$\begin{array}{rl}
d(o, \hat g^{-1}h\cdot o) & \le d(o, to)+d(to, \hat g^{-1} ho)\\
&\le d(o, to)+d(tg\cdot o, ho) \\
\text{(triangle inequality)}&\le d(o, ho)-d(o, go)+2d(t\cdot o, \gamma)+2d(t\cdot ho, \gamma)\\
&\le d(o, ho)-d(o, go)+4\epsilon<d(o, ho),
\end{array}$$this contradicts to the minimal choice of $h$. So the claim is proved and the corollary follows from \ref{GrowthTightThm}. 
\end{proof}

Notice that the set $\mathcal O_M$ is barrier-free. Informally, \ref{GrowthTightThm} can be interpreted as follows: $\mathcal O_M$ is growth-tight if and only if any barrier-free set is growth-tight. On the other hand,  we could produce a barrier-free set without growth-tightness in a geometrically finite group, when  its parabolic gap property fails. In this sense,   \ref{GrowthTightThm} is best possible.

\begin{prop}\label{growthtightsharp}
There exists a divergent group action of $G$ on a geodesic space $(\mathrm Y, d)$ with a contracting element such that for any $\epsilon, M>0$  we have $$
\e {\mathcal V_{\epsilon, M, g}} = \e G
$$ for some $g \in G$.
\end{prop}

\begin{rem}
These groups were constructed by Peign\'e in \cite{Peigne}:  an exotic Schottky group $G$ acts on Cartan--Hadamard manifolds such that $G$ is of divergent type and has no parabolic gap property. A convergent-type case without the parabolic gap property was constructed earlier by Dal'bo et al. \cite{DOP}. 
\end{rem}

\begin{proof}[Sketch of the proof]
Let $G$ be an exotic Schottky group constructed by   Peign\'e in \cite{Peigne2}. It admits  a geometrically finite action of divergent type on a Cartan-Hardamard manifold  and $G$ has no parabolic gap property: there exists a maximal parabolic subgroup $P$ such that $\e P=\e G$.

We draw on a result from \cite[Proposition 1.5]{YANG7}. For $M, \epsilon>0$ fixed, we can invoke the Dehn filling in \cite{DGO} to kill a ``long'' hyperbolic element so that $P$ is ``almostly'' preserved   in the quotient $\bar G$: it projects to be a maximal parabolic subgroup $\bar P$  with $\e {P}=\e{\bar P}$.  Hence, $\e {\bar G}=\e G$. Lift all the elements from $\bar G$ to their shortest representatives in $G$. By the same argument as in Corollary \ref{growthtightcor}, they are contained in $\mathcal V_{\epsilon, M, g}$ so that $\e {\mathcal V_{\epsilon, M, g}} = \e G$. This concludes the proof.
\end{proof}

Recall that a subset $X$ in $\mathrm Y$ is called \textit{weakly $M$-quasi-convex} for a constant $M>0$ if for any two points $x,y$ in $X$ there exists a geodesic $\gamma$ between $x$ and $y$ such that $\gamma \subset N_M(X)$.
\begin{thm}\label{wqcGTight}
Suppose $G$ admits a SCC action on $(\mathrm Y,d)$ with a contracting element. Then every infinite index weakly quasi-convex subgroup $\Gamma$ of $G$ has the property $\e \Gamma <\e G$. 
\end{thm}
\begin{proof}
Let $\epsilon, M$ be the  constants in \ref{GrowthTightThm}, and $M$ is also the quasi-convexity constant of $H$. The idea of proof is to find an element $g \in G$ such that every element $h\in H$ is $(\epsilon, M, g)$-barrier-free. The existence of such $g$ is guaranteed by the following claim.

We claim that for every finite set $F$, the set $G\setminus F\cdot H\cdot F$ is infinite.   Suppose to the contrary that $G\setminus F\cdot H \cdot F$  is finite for some finite $F$. By enlarging $F$ by a finite set,  we can assume that $G=F\cdot H\cdot F$ so $G\subset (\cup_{f\in F} fHf^{-1})F$. This contradicts to a result of  Neumann \cite{Neumann} states that a group $G$ cannot be a finite union of right cosets of infinite index subgroups. Our claim thus follows.

Let $F$ be the set of elements $f\in G$ such that $d(fo, o) \le M+\epsilon$, where  $M$ is the qusiconvexity constant of $H$. Since $G\setminus F\cdot H\cdot F$ is infinite, let us choose one element $g\notin F\cdot H\cdot F$.  In the remainder, we prove  $H \subset \mathcal V_{\epsilon, M, g}$. 

Indeed, suppose to contrary that there exists some $h$ in $H$ which is not  $(\epsilon, M,  g)$-barrier-free. By weak quasi-convexity, there exists a geodesic $\gamma=[o, ho]$ such that $\gamma \subset N_M(Ho)$.  By definition of barriers,  any geodesic between $B_{M}(o)$ and $B_{M}(ho)$ contains an $(\epsilon, g)$-barrier, so for the geodesic $\gamma \subset N_M(Ho)$, there  exist $b\in G$ and $h_1, h_2 \in H$ such that $d(b\cdot o, h_1o), d(b\cdot g o, h_2o)  \le \epsilon+M$.  By definition of $F$, we have $b^{-1}h_1, \; h_2^{-1} b g \in F$. Note that  $g\in b^{-1}h_2F=b^{-1}h_1\cdot h_1^{-1}h_2 F$ and so $g \in F\cdot H\cdot F$ gives a contradiction. The subgroup $H$ is therefore contained in a growth-tight $\mathcal V_{\epsilon, M,g}$, which allows \ref{GrowthTightThm}   to conclude the proof. 
\end{proof}
\begin{rem}
In the proof, we can  choose $g$ to be a contracting element by Lemma \ref{Positivedensity}. 
\end{rem}

We first give a corollary in mapping class groups.
\begin{prop}\label{FMtight}
Any convex-cocompact subgroup $\Gamma$ in  $G\in \mathbb {Mod}$ is growth-tight: $\e \Gamma < \e G.$ 
\end{prop}
\begin{proof}
By \cite[Proposition 3.1]{FarbMosher},  a convex-cocompact subgroup $\Gamma$ in  $G\in \mathbb {Mod}$ is purely pseudo-Anosov and thus is of infinite index: otherwise  any element in $G$ would have some power being pseudo-Anosov. Any orbit of $\Gamma$ on Teichm\"{u}ller space is weakly quasi-convex by \cite[Theorem 1.1]{FarbMosher}.  Since the action of $G$   on the Teichm\"{u}ller space is statistically convex-cocompact \cite{ACTao}, the result thus follows from Theorem \ref{wqcGTight}.
\end{proof}

Here is another corollary for cubulated groups.
\begin{prop}\label{cubconvex}
Suppose that a group   $G$ acts properly and   cocompactly on a CAT(0) cube complex $\mathrm Y$ so that $\mathrm Y$ does not decompose as a product. Then any weakly quasi-convex subgroup of infinite index in $G$ is growth-tight. In particular, any  cubically convex subgroup   is growth-tight, if it is of infinite index.
\end{prop}
\begin{proof}
In \cite[Theorem A]{CapSag}, Caprace and Sageev showed that if $\mathrm Y$ does not decompose as a product, then $G$ contains a rank-1 element which is contracting in our sense. The conclusion therefore follows   from Theorem \ref{wqcGTight}. 
\end{proof}

\section{Purely exponential growth}\label{Section5}
In this section, we first give a proof of \ref{StatisticThm}, and then furnish more details on purely exponential growth of the class of groups listed in Theorem \ref{PEGrowthExamp}.  
\subsection{Proof of   \ref{StatisticThm}} 
We remark that the elementary approach presented here is greatly inspired by the notes of  Peign\`e \cite{Peigne2}. We first recall an elementary lemma. 
\begin{lem}\label{mpeigne}\cite[Fait 1.0.4]{Peigne2}
Given $k>1$, let $a_n$ be a sequence of positive real numbers such that $a_n a_m \le \sum_{|j|\le k} a_{n+m-j}$. Then the following limit   $$\omega:=\lim\limits_{n\to \infty}\frac{\log (a_1+a_2+\cdots+ a_n)}{n}$$ exists and $a_n \prec \exp(n\cdot \omega)$ for $n\ge 1$.
\end{lem}

We now prove the  upper bound for any proper action.
\begin{prop} [Upper bound]\label{expUBD}
Assume that $G$ acts properly on a geodesic metric space $(\mathrm Y,d)$ with a contracting element.  Then the following hold for given $\Delta>0$, 
\begin{enumerate}
\item

$$\e G=\lim\limits_{n\to \infty} \frac{\log \sharp N(o, n, \Delta)}{n}$$
\item
$$
\sharp A(o, n, \Delta)  \prec_\Delta \exp( \e G n)
$$ 
for any $n\ge 1$.
\end{enumerate}
\end{prop}
\begin{proof}
Let $R=R(\epsilon, \Delta)$ be  a constant given by Lemma \ref{extensionmap}. By Lemma \ref{sepnet}, there exist a constant $\theta=\theta(R)>0$ and two subsets $B_1\subset A(n, \Delta)$ and $B_2\subset A(m, \Delta)$ such that  $B_io$ are both $R$-separated and 
\begin{equation}\label{Cseparated}
\theta \cdot \sharp B_1 \ge \sharp A(o, n, \Delta),\; \theta \cdot \sharp B_2 \ge \sharp A(o, m, \Delta).
\end{equation}

The proof proceeds by establishing the following variant of a super-multiplicative inequality: there exists an integer $k>0$ such that 
\begin{equation}\label{supmultipl}
\theta \sharp A(o, n, \Delta) \cdot \theta \sharp A(o, m, \Delta) \le  \theta   \sum_{|j|\le k}\sharp A(o, n+m+j, \Delta).
\end{equation}

For this purpose, we now define a map $\Phi: B_1\times B_2\to G$. Given $W=(b_1, b_2)$, define $\Phi(W)=b_1 fb_2$ for some $f\in F$ provided by the extension lemma \ref{extend3}. Thus by Lemma \ref{extensionmap},  the map $\Phi$ is injective.   Moreover, any geodesic $[o, b_1fb_2o]$ $\epsilon_0$-fellow travels the path labeled by $b_1fb_2$, so   $$d(o, b_1o) +d(o, b_2o) + L\ge d(o, b_1fb_2o)\ge d(o, b_1o) +d(o, b_2o)-2\epsilon_0$$ where $L:=\max\{d(o, fo): f\in F\}<\infty.$ Noting that $B_1\subset A(n, \Delta)$ and $B_2\subset A(m, \Delta)$, we have $|d(o, b_1o)-n|\le \Delta$ and $|d(o, b_2o)-m|\le \Delta$. Setting $k:=\Delta+L+2\epsilon_0$, it then follows $$b_1fb_2 \in A(o, n+m, \Delta+k).$$  Since $\Phi$ is injective, we obtain 
$$
\sharp B_1\cdot \sharp B_2\le \sum_{|j|\le k} \sharp A(o, n+m, \Delta+j),
$$  
with    (\ref{Cseparated}), which establishes  the  inequality (\ref{supmultipl}). Denoting $a_n=\theta\sharp A(o, n, \Delta)$, the proposition then follows from Lemma \ref{mpeigne}.
\end{proof}


We now prove the last claim of   \ref{StatisticThm}.
\begin{thm}[Purely exponential growth]\label{purexpgrowth}
Assume that $G$ admits a SCC action on $\mathrm Y$ with a contracting element. Then there exists $\Delta>0$ such that 
$$
\sharp A(o, n, \Delta)\asymp \exp(\e G n)
$$
for $n\ge 1$. 
\end{thm}
\begin{proof}
By Proposition \ref{expUBD}, it suffices to prove the lower bound. For $\omega>0$, we define: $$b^\omega(n, \Delta)= \exp(-\omega n) \cdot \sharp A(o, n, \Delta).$$ 
The following claim follows by a simpler argument as in the proof of Lemma \ref{SubMulSet}. 
 
\begin{claim}\label{DOPsubmul}
Given $\e G > \omega> \e {\mathcal O_R}$, there exist $\Delta, c_0>0$ such that the following holds 
$$
b^\omega(n+m,\Delta) \le c_0 \left( \sum_{1\le k\le n} b^\omega(k, \Delta) \right) \cdot \left(\sum_{1\le j \le m} b^\omega(j,\Delta) \right),
$$
for any $n, m \ge 0$. 
\end{claim}
 
Denote $V_k:= \sum_{1\le i\le k} b^\omega(i, \Delta)$ and $\tilde V_n:=c_0\cdot \sum_{1\le k\le n} V_k$. By \cite[Lemma 4.3]{DOP}, it is proved that 
$$
\tilde V_{n+m} \le \tilde V_n \tilde V_m
$$
for $n, m\ge 1$. By Feketa's Lemma, it follows that 
$$
\limsup_{n\to \infty} \frac{\log \tilde V_n}{n}=\inf_{n\ge 1} \{\frac{\log \tilde V_n}{n}\}=L
$$
for some $L\in \mathbb R\cup\{-\infty\}$. Take into account the elementary fact   $$\limsup_{n\to \infty} \frac{\log \tilde V_n}{n}=\limsup_{n\to \infty} \frac{\log V_n}{n}=\limsup_{n\to \infty} \frac{\log  b^\omega(n, \Delta)}{n},$$ which  implies that $L=\e { G}-\omega>0$. Hence, $\tilde V_n \ge \exp(Ln)$. 

On the other hand, by Proposition \ref{expUBD}, it follows that $b^\omega(n, \Delta)\prec \exp(Ln)$. From definition of $V_n$ and $\tilde V_n$, it implies that $V_n \prec  \exp(Ln)$ and then $\tilde V_n \prec \exp(Ln)$
for $n\ge 1$. Since $\tilde V_n \ge \exp(Ln)$,   an elementary argument produces a constant $K>0$ such that
$$
\sum_{0\le k <  K} V_{n+k}=\tilde V_{n+K}-\tilde V_{n} \succ \exp(Ln),
$$ 
which yields $V_{n}\succ  \exp(Ln),$ due to $V_n \le V_{n+1}$. We repeat the argument   for $V_{n-K}$ by making use of $V_n \prec  \exp(Ln)$, and show that $  b^\omega(n, \Delta)\succ  \exp(Ln)$. By the fact $L=\e { G}-\omega$, we have $\sharp A(o, n, \Delta) \succ \exp(n\cdot \e G)$, completing the proof.
\end{proof}

Let us record the following useful corollary.
\begin{cor}\label{StatConvexDiv}
Any SCC action with a contracting element is of divergent type: $\PS G$ diverges at $s=\e G$.
\end{cor}

\subsection{Proof of Theorem \ref{PEGrowthExamp}}\label{ProofPEGrowthExamp} We explain in details the proof of each assertion in Theorem \ref{PEGrowthExamp}.

The class of graphical small cancellation groups was  introduced by Gromov \cite{Gro4}  for building exotic  groups (the ``Gromov monster''). In \cite{ACGH}, Arzhantseva et al. have made a careful study of contracting phenomena in  a $\mathrm{Gr}'(1/6)$-labeled graphical small cancellation group $G$.  Such a group $G$ is given by a presentation obtained from a labeled graph $\mathcal G$  such that under a certain small cancellation hypothesis, the graph $\mathcal G$ embeds into the Cayley graph of $G$.  It is proved that if  $\mathcal G$ has only finitely many components  labeled by a finite set $S$, then the action of $G$  on  Cayley graphs contains   a contracting element. Therefore, the following holds by \ref{StatisticThm}: 

\begin{thm}\label{Gr1/6}
A $\mathrm{Gr}'(1/6)$-labeled graphical small cancellation group $G$ with finite components  labeled by a finite set $S$ has purely exponential growth for the corresponding action.
\end{thm} 

The    class of CAT(0) groups with rank-1 elements    admits a geometric (and thus SCC) action with a contracting element.  In particular,  consider the class of a right-angled Artin group (RAAG) whose presentation is obtained from a finite simplicial graph $\Gamma$ as follows:
\begin{equation}\label{RAAG}
G=\langle V(\Gamma)|v_1v_2=v_2v_1 \Leftrightarrow (v_1, v_2)\in E(\Gamma)  \rangle
\end{equation}
See \cite{Kob} for a reference on RAAGs. It is known that an RAAG acts properly and cocompactly on a non-positively curved  cube complex called the \textit{Salvetti complex}.  

 It is known that the defining graph is a join if and only if the RAAG is a direct product of non-trivial groups.  In \cite[Theorem 5.2]{BehC}, Behrstock and Charney proved that any subgroup of an  RAAG $G$   that is not conjugated into a \textit{join subgroup} (i.e., obtained from a join subgraph)  contains a  contracting element.  Therefore, we obtain the following:

\begin{thm}[$\mathbb {RAAG}$]
Right-angled Artin groups that are not direct products have purely exponential growth for the action on their Salvetti complex. 
\end{thm}

The class of right-angled Coxeter groups (RACGs) can be defined as in (\ref{RAAG})   with additional relations $v^2=1$ for each $v\in V(\Gamma)$. An RACG also acts properly and cocompactly on a CAT(0) cube complex called the Davis complex (which is equal to the Salvetti complex of the corresponding RAAG). In \cite[Proposition 2.11]{BHS}, Berhstock et al. characterized   an RACG $G$ of linear divergence as virtually a direct product of groups.   By \cite[Theorem 2.14]{ChaSul}, Charney and Sultan proved that the existence of rank-1 elements is equivalent to a superlinear divergence.  Hence, we have the following.
\begin{thm}[$\mathbb {RACG}$]\label{Behrstock} 
If a right-angled Coxeter group is not virtually a direct product of non-trivial groups, then it has purely exponential growth for the action on the Davis complex. 
\end{thm}

\section{Constructing SCC actions}\label{Section6}
In this   section, we shall present a simple method to produce a statistically convex-cocompact action. The main result is constructing examples in $\mathbb {Mod}$ of irreducible subgroups with a  SCC action on Teichm\"{u}ller space but which are not convex-cocompact.
\subsection{Independence of basepoints for SCC actions}
In this subsection we show that SCC actions with a contracting element do not depend on the choice of basepoints. The ingredient of the proof is the growth-tightness   \ref{GrowthTightThm}. It is not clear whether the assumption of the existence of a contracting element is removable.  
\begin{lem}\label{concaveRegion}
If $\e {\mathcal O_{M}}<\e G$ for some $M>0$, then  $\e {\mathcal O_{M_1, M_2}}<\e G$ for any $M_2\ge  M_1 \gg M$.
\end{lem} 
\begin{proof}
Fix a contracting element $f\in G$. By \ref{GrowthTightThm}, the set $\mathcal V_{\epsilon, M, f^m}$ is growth-tight for any $m>0$. The proof consists in verifying $\mathcal O_{M_1, M_2}\subset \mathcal V_{\epsilon, M, f}$ for appropriate constants $M_1, M_2$. 

Let $g\in \mathcal O_{M_1, M_2}$ so there exists a geodesic $\gamma$ between $B(o, M_2)$ and $B(go, M_2)$ such that the interior of $\gamma$ lies outside $N_{M_1}(Go)$. Assume, to the contrary, that $g\notin \mathcal V_{\epsilon, M, f^m}$, then any geodesic $\beta$ between $B(o, M)$ and $B(go, M)$ has an $(\epsilon ,f^m)$-barrier so there exists $b\in G$ such that $$\diam{b\ax(f)\cap \beta}\ge d(o, f^mo)-2\epsilon$$ where the right side tends to $\infty$ as $m\to \infty$. Denote by $C$ the contraction constant of $\ax(f)$. Noticing that $d(\gamma_-, \beta_-), d(\gamma_+, \beta_+)\le M+M_2,$ a contracting argument implies that if $d(o,f^{m}o)$ is sufficiently large comparable with $ M+M_2$, then  $\gamma\cap N_C(b\ax(f))\ne \emptyset.$ By setting $M_1>C$, we obtain   $\gamma\cap N_{M_1}(Go)\ne \emptyset$,  a contradiction with the assumption of $\mathring{\gamma}$ outside $N_{M_1}(Go)$. Hence,  it is proved that $\mathcal O_{M_1, M_2}\subset \mathcal V_{\epsilon, M, f^m}$ for large $m$: the proof is done.
\end{proof}

\begin{lem}\label{SCCBasepoint}
The definition of a SCC action with a contracting element is independent of the choice of basepoints. 
\end{lem} 
\begin{proof}
Consider different basepoints $o, o' \in \mathrm Y$. We choose two constants $M_1\le M_2$ by Lemma \ref{concaveRegion}  such that $M_2>M_1+d(o,o')$ and $\e {\mathcal O_{M_1, M_2}}<\e G$. Define   $M_1'=M_1+d(o, o')$ and $M_2'=M_2-d(o, o')$. We claim that $\mathcal O'_{M_1', M_2'}\subset \mathcal O_{M_1, M_2}$, where $\mathcal O'_{M_1', M_2'}$ is the concave region defined using the basepoint $o'$.  

Indeed, let $g \in   \mathcal O'_{M_1', M_2'}$ so some geodesic $[x',y']$ between $B(o', M_2')$ and $B(go', M_2')$  has the interior disjoint with $N_{M_1'}(Go')$. Since $M_1'= M_1+d(o, o')$, we have $N_{M_1}(Go) \subset N_{M_1'}(Go')$ so the interior of $[x', y']$   lies outside $N_{M_1}(Go)$ as well.  By the choice of $M_2'=M_2-d(o, o')$, the geodesic $[x', y']$   lies between  $B(o, M_2)$ and $B(go, M_2)$. Thus,   $g \in \mathcal O_{M_1, M_2}$ so the claim thus follows. Thus,  $\e {\mathcal O_{M'}} <\e G$ by Lemma \ref{concaveRegion},  concluding the proof of lemma.
\end{proof}

\subsection{Free product combination}
 With the critical gap criteron \ref{degrowth}, the following combination result is not surprising.

\begin{prop}\label{FreeproductComb}
Assume that $G$ acts properly on a geodesic metric space $(\mathrm Y, d)$. Consider two  subgroups $H, K$ such that $H$ is a residually finite, contracting subgroup and $\proj_{Ho}(Ko)<\infty$ for a basepoint $o \in \mathrm Y$.

If  either $K$ is residually finite, or $\{k H\cdot o: k\in K\}$ has bounded intersection, then there exist finite index subgroups $\hat H$, $\hat K$  of $H, K$ respectively such that the subgroup $\Gamma$ generated by $\hat H, \hat K$ is isomorphic to $\hat H\star \hat K$.   

In addition, if $\e H\ge\e K$ or $\e H=\e K=0$, then $\Gamma$ admits a SCC action on $\mathrm Y$.
\end{prop}
\begin{proof}
By hypothesis, $\diam{\pi_{Ho}(Ko)} \le \tau$ for some constant $\tau>0$. Let $C$ be the contraction constant of the subset $Ho$. Choose a large constant $D>\max\{D(\tau),  c(\tau)\}$ where $D(\tau), c(\tau)$ are constants given by Proposition \ref{admissible}. Since $H$ is residually finite and the action is proper,  there exists a finite index subgroup $\hat H$ such that $\hat H\cap   K=\{1\}$ and $d(o, ho)>D$ for any $1\ne h\in \hat H$. In the second case of the assumptions on $K$, we just let $\hat K=K$, otherwise, since $K$ is residually finite,   we choose a finite index subgroup $\hat K$ such that $d(o, ko)>D$ for any $1\ne k\in \hat K$. We are now going to prove that $\langle \hat H, \hat K\rangle = \hat H\star \hat K$. 

Indeed, it suffices to establish that each non-trivial word $W$ with letters alternating in $\hat H$ and $\hat K$ labels a $(D, \tau)$-admissible path so a $c(\tau)$-quasi-geodesic. To be precise, write $W=h_1k_1\cdots h_ik_i \cdots h_nk_n$ where $h_i\in \hat H, k_i\in \hat K$ and $h_1, k_n$ may be trivial. Let $\gamma$ be the path labeled by $W$ for which the system of contracting sets is given by $Ho, k_1Ho, \cdots, h_1\cdots k_{n-1} Ho$.
Note first that  the conditions (\textbf{LL1})  and (\textbf{BP}) are clear by the choice of $\hat H$ as above. Two possibilities in the condition (\textbf{LL2}) correspond to the two facts that either $\mathbb X=\{k H\cdot o: k\in K\}$ has bounded intersection, or $d(o, ko)>D$ for $k\in \hat K$.   So $\gamma$ is a $(D, \tau)$-admissible path   of length at least $D> c(\tau)$ and, since it is  a $c(\tau)$-quasi-geodesic, we obtain by computation that $h_1k_1\cdots h_ik_i \cdots h_nk_n\ne 1 \in \langle \hat H,\hat K\rangle$. This proves that $\langle \hat H, \hat K\rangle = \hat H\star \hat K$.

For sufficiently large $D$, the goal is to prove that the concave region $\mathcal O_{C}$ belongs to $\hat K$. Consider an element $g=g=h_1k_1\cdots h_ik_i \cdots h_nk_n \in \Gamma\setminus \hat K$, which labels an admissible path $\gamma$.   Since $g\in \Gamma\setminus\hat K$, there must exist a contracting set $X$ from $\mathbb X$ associated to a subpath $p$ labeling an element $h_i$   such that, for a constant $B=B(\tau)$ given by Proposition \ref{admissible}, we have $\proj_X(\gamma_1 \cup \gamma_2)\le 2B$ where $\gamma_1:=[\gamma_-, p_-]_\gamma, \gamma_2:=[p_+, \gamma_+]_\gamma$ are subpaths of $\gamma$ before and after $X$.

Let $\alpha$ be a geodesic between $B(o, C)$ and $B(go, C)$, so $d(\gamma_-, \alpha_-), d(\gamma_+, \alpha_+)\le C$.  By Proposition \ref{Contractions}.\ref{1Lipschitz}, we have $\max\{\proj_X(\gamma_-, \alpha_-), \proj_X(\gamma_+, \alpha_+)\}\le 2C.$
Choose further $D>7C+2B,$ we see that $N_C(X)\cap \alpha \ne \emptyset.$ Indeed, if $N_C(X)\cap \alpha = \emptyset$ so $\proj_X(\alpha)\le C$, we obtain by a projection argument that
$$
\begin{array}{lll}
\len(p) &\le d(p_-, X)+\proj_X(\gamma_1 \cup \gamma_2)+d(p_+, X)+ \proj_X(\alpha)+\proj_X(\gamma_-, \alpha_-)+ \proj_X(\gamma_+, \alpha_+)\\
& \le C+2B+C+2C+C+2C\le 7C+2B,
\end{array}
$$
which is a contradiction to the choice of $D>7C+2B$ for $\len(p)\ge D$. Thus,  $N_C(X)\cap \alpha \ne \emptyset$ is proved. Since $X$ is a translate of $Ho$ so $X\subset Go$,  it follows that $\alpha\cap N_C(Go)\ne \emptyset:$  $g$ does not belong to $\mathcal O_C$. Hence, $\mathcal O_C \subset \hat K$ is proved. 

We next show that the action is SCC, provided that $\e H>\e K$ or $\e H=\e K=0$. 
For a contracting subgroup, any orbit is quasi-convex by Propoistion \ref{Contractions} so $\hat H$ acts by a SCC action: its Poincar\'e series $\PS {\hat H}$ diverges at $s=\e {\hat H}$ by Corollary \ref{StatConvexDiv}. Consider the case   $\e H \ge \e K$. Since the growth rate remains the same for finite index subgroups so $\e H=\e {\hat H}$ and $\e K= \e {\hat K}$, it follows by Lemma \ref{degrowth} that $\e \Gamma\ge\e {\hat H}>\e {\hat K}$. Hence, $\Gamma$ admits a SCC action. For the case $\e H= \e K=0$, since $\Gamma$ contains non-abelian free subgroups so $\e \Gamma>0$,   the action of $\Gamma$   is SCC as well. The proof is thus complete.
\end{proof}

\subsection{Mapping class groups}
In mapping class groups, we use a boundary separation argument to fullfill the criterion of Proposition \ref{FreeproductComb}. This idea is well-known in the setting of (relatively) hyperbolic groups.  In what follows, we explain how to implement it in $\mathbb {Mod}$ using Thurston boundary. The references are \cite{FLP}, \cite{FMbook} and in particular, the theory of limit sets in mapping class groups developped in \cite{McPapa}.

Recall that Teichm\"{u}ller space $(\mathrm Y, d)$ of a closed orientable surface $\Sigma$ with negative Euler characteristic is compactified by the space $\mathcal{PMF}$ of \textit{projective measured laminations} on $\Sigma$, which is the  quotient of \textit{measured laminations}  $\mathcal {MF}$ by positive reals. The proper action on $\mathrm Y$ of the mapping class group $G$    extends to $\mathcal{PMF}$  by homeomorphisms, which is called \textit{Thurston boundary} of $\mathrm Y$. 

There exists non-zero, symmetric, $G$-invariant,  bi-homogeneous intersection function $$i: \mathcal {MF}\times \mathcal {MF} \to \mathbb R_{\ge 0}.$$ 

Let $\Lambda$ be a subset    in $\mathcal{PMF}$. The \textit{intersection completion} $Z(\Lambda)$ consists of projective measured laminations $[G]\in \mathcal{PMF}$ such that $i(F, G)=0$ for some $[F]\in \Lambda$.  A \textit{uniquely ergodic} point $[F]$ in $\mathcal{PMF}$ has the property that if $i(G, F)=0$ then $[G]= [F]$. 
 
According to \cite{McPapa},  the limit set of a group action on a topological space is the set of accumulation points of  all orbits. In this regard, the \textit{enlarged limit set} $\Lambda^e(\Gamma)$ (resp. \textit{limit set} $\Lambda(\Gamma)$) of a subgroup $\Gamma$ is the set of accumulation points of  all $\Gamma$-orbits in $\mathcal{PMF}\cup \mathrm Y$ (resp. $\mathcal{PMF}$). By Proposition 8.1 in \cite{McPapa}, it follows that $\Lambda^e(\Gamma) \subset \Lambda(\Gamma)$.

An infinite \textit{reducible} subgroup $H$ preserves   a finite family of disjoint simple closed curves called a \textit{reduction system} on $\Sigma$.  The following elementary fact will be useful later on.
\begin{lem}\label{disjointlimitsets}
The enlarged limit set of an infinite reducible subgroup $H$ is disjoint with that of the subgroup $P$ generated by a pseudo-Anosov element $p$.
\end{lem} 
\begin{proof}
By \cite[Section 7.1]{McPapa}, the limit set $\Lambda(H)$ of $H$ is the union of an essential  reduction system $A$ with projective measured laminations on pseudo-Anosov components obtained by cutting $A$ on $\Sigma$.   As a consequence of this description, all points in $\Lambda(H)$ are non-filling. On the other hand, the limit set $\Lambda(P)$ of $P$ consists of two filling uniquely ergodic points so $Z(\Lambda(P))=\Lambda(P)$.  Therefore, $\Lambda(H) \cap \Lambda(P)=\emptyset$ and thus $Z(\Lambda(H)) \cap Z(\Lambda(P))=\emptyset$.
\end{proof}

The proof of the next result makes use of the following fact  in \cite[Lemma   1.4.2]{KaMasur}. 
Suppose that $x_n \in \mathrm Y$ is a sequence of points converging to a uniquely ergodic point $\xi \in \mathcal {PMF}$. If $$d(o,y_n)-d(x_n,y_n)\to +\infty$$ for a sequence $y_n \in \mathrm Y$, then $y_n \to \xi$.

\begin{lem}\label{ModBP}
  For any basepoint $o\in \mathrm Y$, there exists a constant $D>0$ such that $Ko$ have a $D$-bounded projection to $Ho$.
\end{lem}
\begin{proof}
Suppose to the contrary that there exists $k_n\in K$ such that $\proj_{Ho}([o, k_no])\to \infty$.  Let $x_n \in Ko$ be a projection point of $y_n:=k_no$ to $Po$. Since $d(o, x_n)\to \infty$, up to passage to subsequence, it converges to a limit point in $Z(\Lambda(P))=\Lambda(P)$. 

On the other hand, the contracting property shows that $d(o, y_n)-d(x_n,y_n)$ differs from $d(o, x_n)$ up to a uniform bounded amount. Hence, by Lemma   1.4.2 in \cite{KaMasur}, one  concludes that $y_n$ and $x_n$ converge to the same limit point, giving a contradiction to Lemma \ref{disjointlimitsets}. The proof is thus complete.
\end{proof}

We are ready to prove the following main result of this section.
\begin{prop}\label{ModSubSCC}
Let $p$ be a pseudo-Anosov element and $K$ be an infinite torsion-free reducible subgroup in a mapping class group. There exists $n>0$ such that $\langle p^n, K\rangle$ is a free product of $\langle p^n\rangle$ and $K$ acting on the Teichm\"{u}ller space via a SCC action.
\end{prop}
\begin{proof}
We apply Proposition  \ref{FreeproductComb} to $H=E(p)$ and $K$, where $E(p)$ defined in (\ref{Ehdefn}) is an elementary contracting subgroup with bounded intersection by Lemma \ref{elementarygroup}. The only reducible elements in $E(p)$ are torsions so $E(p)\cap K=\{1\}.$ Hence, the collection $\{k E(p)\cdot o: k\in K\}$ has bounded intersection. By Proposition  \ref{FreeproductComb}, the conclusion thus follows from Lemma \ref{ModBP}.
\end{proof}

Two examples deriving from Proposition \ref{ModSubSCC} are as follows:
\begin{enumerate}
\item
Let $p$ be pseudo-Anosov and $k$   an infinite reducible element in $\mathbb {Mod}$. Then $\langle p^n, k\rangle$ admits a SCC action for $n\gg 0$.
\item
Let $p$ be pseudo-Anosov and $K$ be an abelian group generated by Dehn twists around disjoint essential simple closed curves. Then $\langle p^n, K\rangle$ admits a SCC action for $n\gg 0$.
\end{enumerate}

\bibliographystyle{amsplain}
 \bibliography{bibliography}

\end{document}